\documentclass[11pt,a4paper]{article}
\usepackage{mycommands}

\newcommand{\smallT}{\text{\small$\mathbf{T}$}}
\newcommand{\subT}{\mathbf{T}}
\newcommand{\smallF}{\text{\small$\mathbf{F}$}}
\newcommand{\subF}{\mathbf{F}}

\title{\textbf{Permutations avoiding 1324 \\
and patterns in {\L}uka\-sie\-wicz paths}}

\author{$\phantom{{}^\dagger}$David Bevan${}^\dagger$}

\date{}

\begin{document}
\maketitle

{\let\thefootnote\relax\footnotetext
{${}^\dagger$Department of Mathematics and Statistics, The Open University, Milton Keynes, England.}}

{\let\thefootnote\relax}

\begin{abstract}
\noindent
The class $\av(\pdiamond)$, of permutations avoiding the pattern $\pdiamond$,
is one of the simplest sets of combinatorial objects to define
that has, thus far, failed to reveal its enumerative secrets.
By considering certain large subsets of the class,
which consist of permutations with a particularly regular structure,
we prove that the growth rate of the class exceeds $9.81$.
This improves on a previous lower bound of $9.47$.
Central to our proof is an examination of the asymptotic
distributions of certain substructures in
the Hasse graphs of the permutations.
In this context,
we consider
occurrences of patterns in {\L}uka\-sie\-wicz paths and prove that
in the limit they exhibit a concentrated Gaussian distribution.
\end{abstract}

\section{Introduction}\label{sectIntro}

We identify a permutation with the sequence of its values.
A permutation $\sigma=\sigma_1\ldots\sigma_n$ of $\{1,\ldots,n\}$ is said to \emph{avoid} a
permutation (often referred to as a \emph{pattern})
$\pi=\pi_1\ldots\pi_k$ of $\{1,\ldots,k\}$ if there is no subsequence of $\sigma$ that has the same relative order as $\pi$.
The class consisting of those permutations that avoid a permutation $\pi$ is denoted by $\av(\pi)$.
Due to the celebrated proof of the Stanley--Wilf conjecture by Marcus \& Tardos~\cite{MT2004}, it is known that $\av(\pi)$
has a finite asymptotic \emph{growth rate}
$$
\gr(\av(\pi)) \;=\; \liminfty \sqrt[n]{S_n(\pi)},
$$
where $S_n(\pi)$ is the number of elements of $\av(\pi)$ of length $n$.
The growth rate of $\av(\pi)$ is also known as the \emph{Stanley--Wilf limit} of $\pi$.

Our interest is in $\av(\pdiamond)$.
This is the only class avoiding a pattern of length four that is yet to be enumerated exactly. Moreover, even the growth rate of the $\pdiamond$-avoiders is currently unknown.
In a recent paper,
Conway \& Guttmann~\cite{CG2015}
calculate the number of permutations avoiding $\pdiamond$ up to length
$36$, building on earlier work by Johansson \& Nakamura~\cite{JN2014}.
They then analyse the sequence of values and
give an estimate for
the growth rate of $\av(\pdiamond)$
of~$11.60\pm0.01$.
However, rigorous bounds still differ from this value quite markedly.

The last few years have seen a steady reduction
in upper bounds on the growth rate, 
based on
a
colouring scheme of Claesson, Jel\'inek \& Stein\-gr\'imsson~\cite{CJS2012} which yields a 
value of~$16$.
B\'ona~\cite{Bona2014+}
has now reduced this to
$13.73718$ by employing
a refined counting argument.

As far as lower bounds go,
Albert, Elder, Rechnitzer, Westcott \& Zabrocki~\cite{AERWZ2006}
have
established that the growth rate is at least $9.47$,
by
using the \emph{insertion encoding} 
of
$\pdiamond$-avoiders
to construct a
sequence of finite automata that accept subclasses of $\av(\pdiamond)$.
The growth rate of a subclass is then determined from the
transition matrix of the corresponding automaton.
Our main result is an improvement on this lower bound:

\thmbox{
\begin{thm}\label{thm1324LowerBound}
$\gr(\av(\pdiamond)) > 9.81$.
\end{thm}
}

\begin{figure}[ht]
  $$
  \begin{tikzpicture}[scale=0.2,line join=round]
    \draw [black!50,very thick] (1,15)--(4,20);
    \draw [black!50,very thick] (2,9)--(3,10)--(5,13)--(15,14);
    \draw [black!50,very thick] (7,5)--(12,6)--(13,8);
    \draw [black!50,very thick] (14,1)--(19,3);
    \draw [black!50,very thick] (1,15)--(6,16)--(8,19);
    \draw [black!50,very thick] (10,18)--(6,16)--(11,17);
    \draw [black!50,very thick] (3,10)--(16,11)--(17,12);
    \draw [black!50,very thick] (12,6)--(18,7);
    \draw [black!50,very thick] (14,1)--(20,2);
    \draw [black!50,very thick] (3,10)--(4,20);
    \draw [black!50,very thick] (5,13)--(6,16);
    \draw [black!50,very thick] (8,19)--(7,5)--(10,18)--(9,4)--(11,17)--(7,5);
    \draw [black!50,very thick] (9,4)--(12,6);
    \draw [black!50,very thick] (18,7)--(14,1)--(16,11)--(13,8)--(15,14)--(14,1);
    \plotpermnobox{20}{15, 9, 10, 20, 13, 16, 5, 19, 4, 18, 17, 6, 8, 1, 14, 11, 12, 7, 3, 2}
  \end{tikzpicture}
  \qquad\qquad
  \begin{tikzpicture}[scale=0.2,line join=round]
    \draw [black!50] (1,15)--(4,20);
    \draw [black!50] (2,9)--(3,10)--(5,13)--(15,14);
    \draw [black!50] (7,5)--(12,6)--(13,8);
    \draw [black!50] (14,1)--(19,3);
    \draw [black!50] (1,15)--(6,16)--(8,19);
    \draw [black!50] (10,18)--(6,16)--(11,17);
    \draw [black!50] (3,10)--(16,11)--(17,12);
    \draw [black!50] (12,6)--(18,7);
    \draw [black!50] (14,1)--(20,2);
    \draw [black!50] (3,10)--(4,20);
    \draw [black!50] (5,13)--(6,16);
    \draw [black!50] (8,19)--(7,5)--(10,18)--(9,4)--(11,17)--(7,5);
    \draw [black!50] (9,4)--(12,6);
    \draw [black!50] (18,7)--(14,1)--(16,11)--(13,8)--(15,14)--(14,1);
    \draw [red,very thick] (10,18)--(9,4)--(11,17);
    \draw [red,very thick] (9,4)--(12,6)--(13,8)--(15,14);
    \draw [red,very thick] (12,6)--(18,7);
    \draw [red,very thick] (13,8)--(17,12);
    \draw [thin] (9,4) circle [radius=0.45];
    \plotpermnobox{20}{15, 9, 10, 20, 13, 16, 5, 19, 4, 18, 17, 6, 8, 1, 14, 11, 12, 7, 3, 2}
  \end{tikzpicture}
  \qquad\qquad
  \begin{tikzpicture}[scale=0.2,line join=round]
    \draw [black!50] (1,15)--(4,20);
    \draw [black!50] (2,9)--(3,10)--(5,13)--(15,14);
    \draw [black!50] (7,5)--(12,6)--(13,8);
    \draw [black!50] (14,1)--(19,3);
    \draw [black!50] (1,15)--(6,16)--(8,19);
    \draw [black!50] (10,18)--(6,16)--(11,17);
    \draw [black!50] (3,10)--(16,11)--(17,12);
    \draw [black!50] (12,6)--(18,7);
    \draw [black!50] (14,1)--(20,2);
    \draw [black!50] (3,10)--(4,20);
    \draw [black!50] (5,13)--(6,16);
    \draw [black!50] (8,19)--(7,5)--(10,18)--(9,4)--(11,17)--(7,5);
    \draw [black!50] (9,4)--(12,6);
    \draw [black!50] (18,7)--(14,1)--(16,11)--(13,8)--(15,14)--(14,1);
    \draw [blue,very thick] (1,15)--(6,16)--(5,13)--(3,10)--(2,9);
    \draw [blue,very thick] (6,16)--(8,19)--(7,5);
    \draw [thin] (8,19) circle [radius=0.45];
    \plotpermnobox{20}{15, 9, 10, 20, 13, 16, 5, 19, 4, 18, 17, 6, 8, 1, 14, 11, 12, 7, 3, 2}
  \end{tikzpicture}
  $$
  \caption{The Hasse graph of a $\pdiamond$-avoider, with an up-set and a down-set highlighted}
  \label{figHasseGraphExample}
\end{figure}
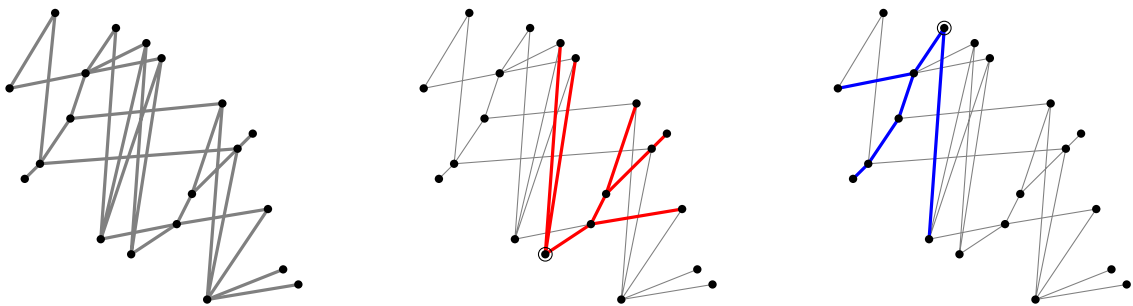
To each permutation $\sigma$, we associate a plane graph $H_\sigma$, which we call its \emph{Hasse graph}.
To create the Hasse graph for a permutation $\sigma=\sigma_1\ldots\sigma_n$, let
vertex $i$ be the point $(i,\sigma_i)$ in the Euclidean plane.
Now, for each pair $i,j$ such that $i<j$, add an edge between vertices $i$ and $j$, if and only if $\sigma(i)<\sigma(j)$ and there is no vertex $k$ such that $i<k<j$ and $\sigma(i)<\sigma(k)<\sigma(j)$.
See Figure~\ref{figHasseGraphExample} for an example.
Note that the edges of $H_\sigma$ correspond to the edges of the Hasse diagram of the sub-poset, $P_\sigma$, of $\mathbb{N}^2$ consisting of the points $(i,\sigma_i)$.
{Hasse graphs of permutations were previously considered by Bousquet-M\'elou \& Butler~\cite{BMB2007} who
determined the algebraic generating function of the class of \emph{forest-like} permutations
whose Hasse graphs are acyclic.}

If a permutation avoids the pattern $\pdiamond$, then its
Hasse graph does not have the diamond graph
$H_\pdiamond =
\raisebox{-2.5pt}{\begin{tikzpicture}[scale=0.12,line join=round]
  \draw[] (1,1)--(2,3)--(4,4)--(3,2)--(1,1);
  \plotpermnobox{}{1,3,2,4}
\end{tikzpicture}}
$
as a minor.
In particular, both up-sets and down-sets of the poset $P_\sigma$ are \emph{trees}.
In other words,
the subgraph of $H_\sigma$
induced by a \emph{left-to-right minimum} of $\sigma$ (a minimal element in the poset $P_\sigma$) and the points to its north-east is a tree, 
as is that
induced by a \emph{right-to-left maximum} of $\sigma$ (a maximal element in the poset $P_\sigma$) and the points to its south-west.
See Figure~\ref{figHasseGraphExample} for an illustration.

\begin{figure}[t]
\begin{center}
  \includegraphics[scale=0.6]{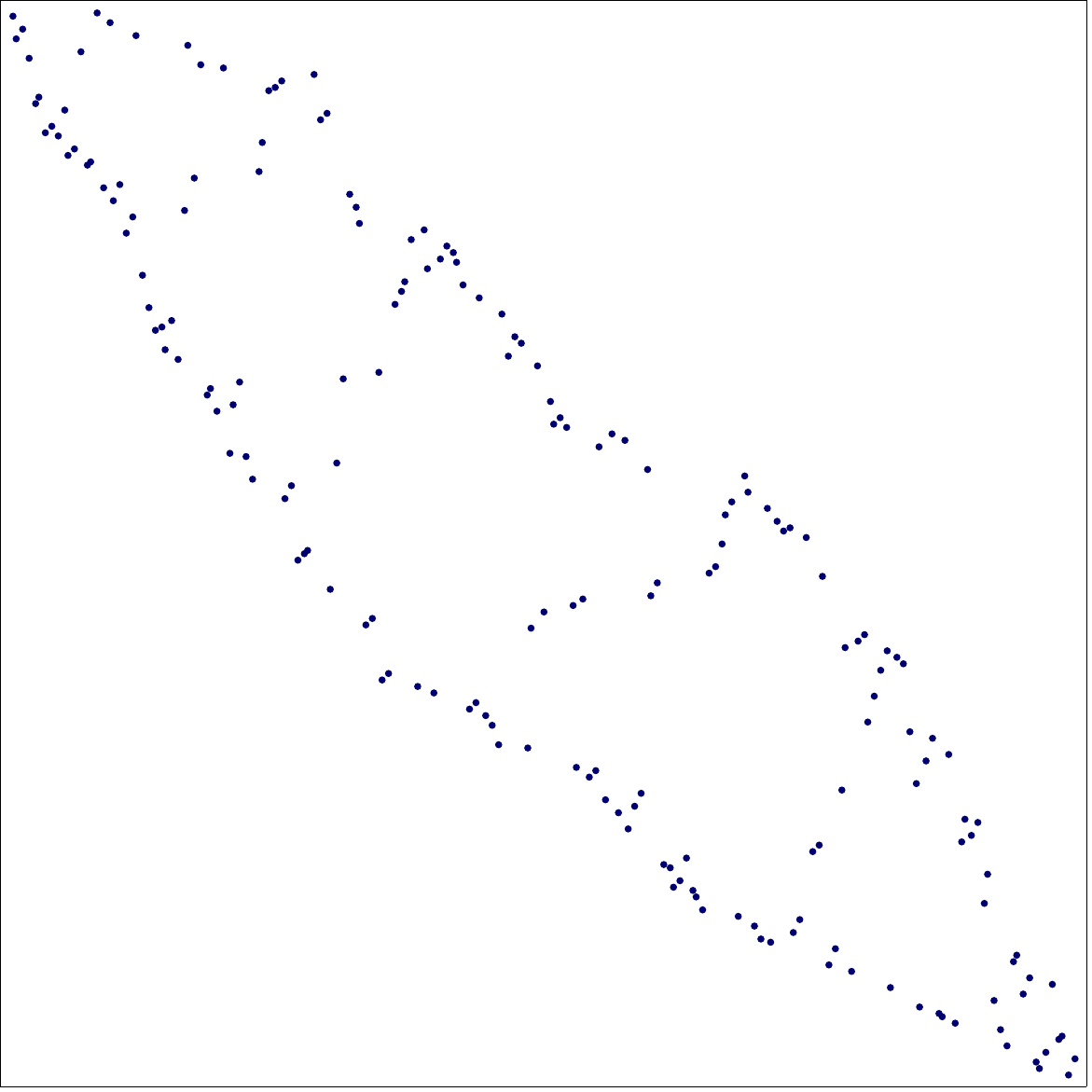}
  $
  \qquad
  $
  \includegraphics[scale=0.6]{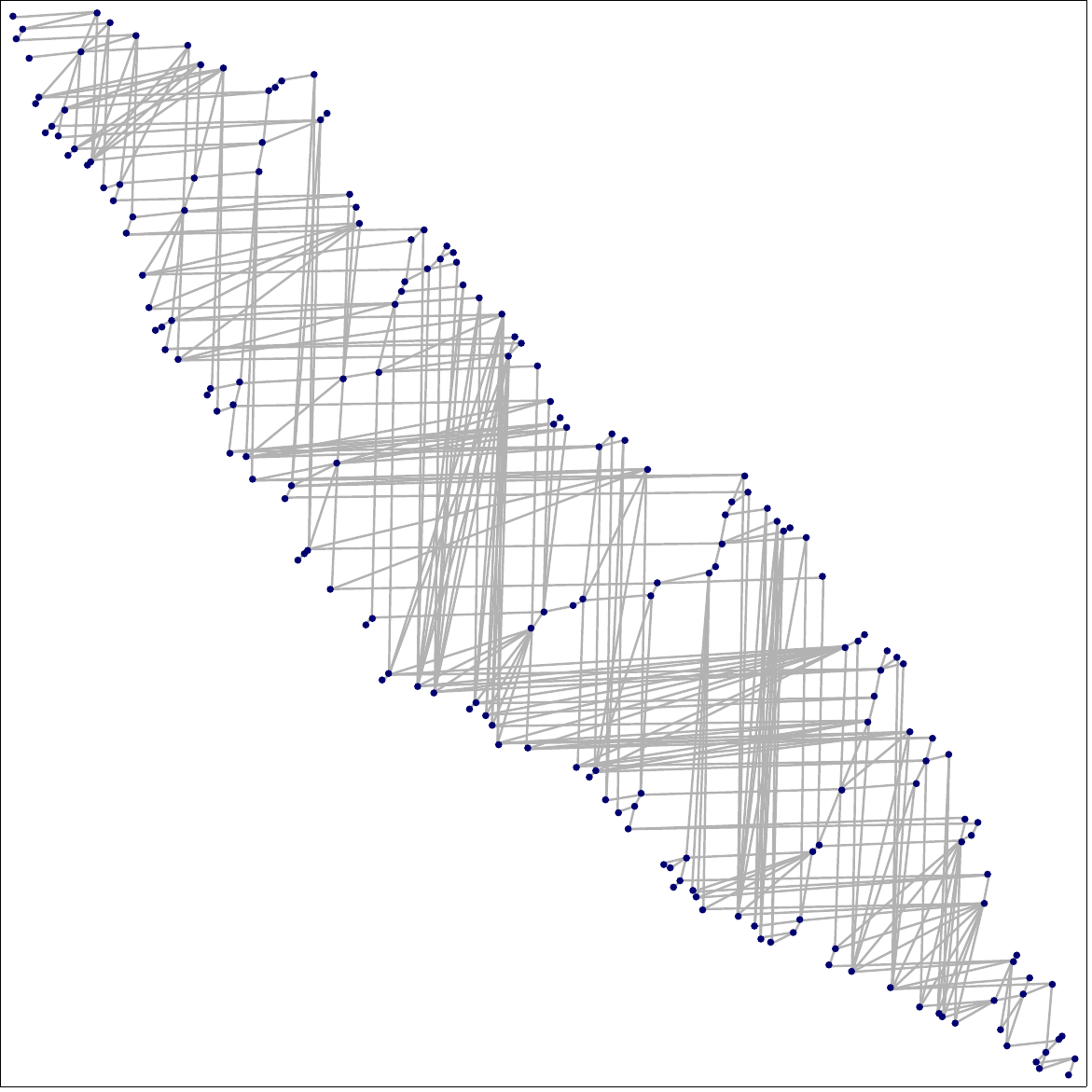}
\end{center}
  \caption{The plot of a $\pdiamond$-avoider of length 187 and its Hasse graph}
  \label{fig187}
\end{figure}

What does a \emph{typical} $\pdiamond$-avoider look like?
Figures~\ref{fig187} and~\ref{figEinar1000} contain illustrations 
of large $\pdiamond$-avoiders.\footnote{The data for Figure~\ref{figEinar1000} was provided by Einar Steingr\'imsson from the investigations he describes in~\cite{Steingrimsson2013} Footnote~4.} As is noted by Flajolet \& Sedgewick (\cite{FS2009} p.169), the fact that a single example can be used to illustrate the asymptotic structure of a large random combinatorial object can be attributed to concentration of distributions, of which we make much use below in determining our lower bound.
Observe the cigar-shaped boundary regions consisting of numerous small subtrees, and also the relative scarcity of points in the interior, which tend to be partitioned into a few paths connecting the two boundaries.
Many questions concerning the
shape of a typical large $\pdiamond$-avoider
remain to be answered or even to be posed precisely.
The recent investigations of
Madras \& Liu~\cite{ML2010} and Atapour \& Madras~\cite{AM2014}
provide a starting point.

We will be restricting our attention to
$\pdiamond$-avoiders whose Hasse graphs are spanned by a disjoint sequence of trees, rooted at alternate boundaries.
In our investigation of how these trees can interact, we consider the asymptotic distribution of certain substructures of the Hasse graphs.
In doing so, we
exploit the fact that plane trees
are in bijection with \emph{{\L}uka\-sie\-wicz paths}.
A {\L}uka\-sie\-wicz path of length $n$ is a
sequence of integers $y_0,\ldots,y_n$ such that $y_0=0$,
$y_i\geqslant1$ for $i\geqslant1$,
and each \emph{step} $s_i=y_i-y_{i-1}\leqslant1$.
Thus, at each step, a {\L}uka\-sie\-wicz path may rise by at most one, but may fall by any amount as long as it doesn't drop to zero or below.

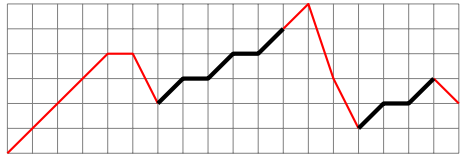
\begin{figure}[ht]
  $$
  \begin{tikzpicture}[scale=0.33,line join=round]
    \draw[help lines] (0,0) grid (18,6);
    \draw[red,thick] (0,0)--(4,4)--(5,4)--(6,2);
    \draw[red,thick] (11,5)--(12,6)--(13,3)--(14,1);
    \draw[red,thick] (17,3)--(18,2);
    \draw[ultra thick] (6,2)--(7,3)--(8,3)--(9,4)--(10,4)--(11,5);
    \draw[ultra thick] (14,1)--(15,2)--(16,2)--(17,3);
  \end{tikzpicture}
  $$
  \caption{The plot of a {\L}uka\-sie\-wicz path that contains three
  occurrences of the pattern $\mathbf{1,0,1}$, two of which overlap}
  \label{figLukaPath}
\end{figure}
In particular, we investigate the distribution of \emph{patterns} in {\L}uka\-sie\-wicz paths.
A pattern $\omega$ of length $m$ in such a path
is a sequence of steps $\omega_1,\ldots,\omega_m$
that occur contiguously in the path (i.e. there is some $k\geqslant0$ such that $\omega_j=s_{k+j}$ for $1\leqslant j\leqslant m$),
with the restriction that the \emph{height}
$\sum_{j=1}^i \omega_j$ after the $i$th step is positive
for $1\leqslant i\leqslant m$.
Note that multiple occurrences of a given pattern may overlap in a {\L}uka\-sie\-wicz path. See Figure~\ref{figLukaPath} for an illustration.

\begin{figure}[htp]
\begin{center}
  \includegraphics[scale=0.38]{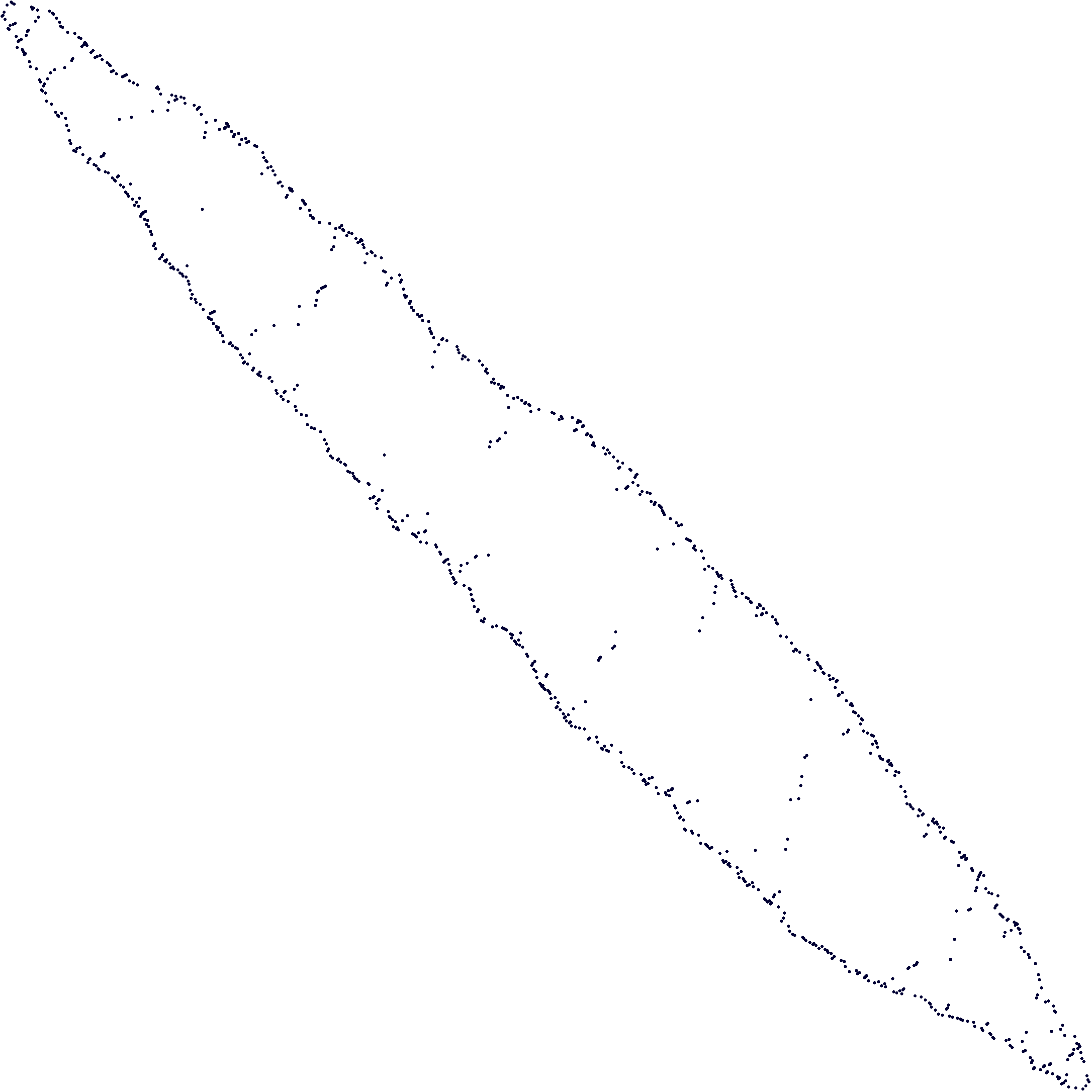}
\end{center}
\begin{center}
  \includegraphics[scale=0.36]{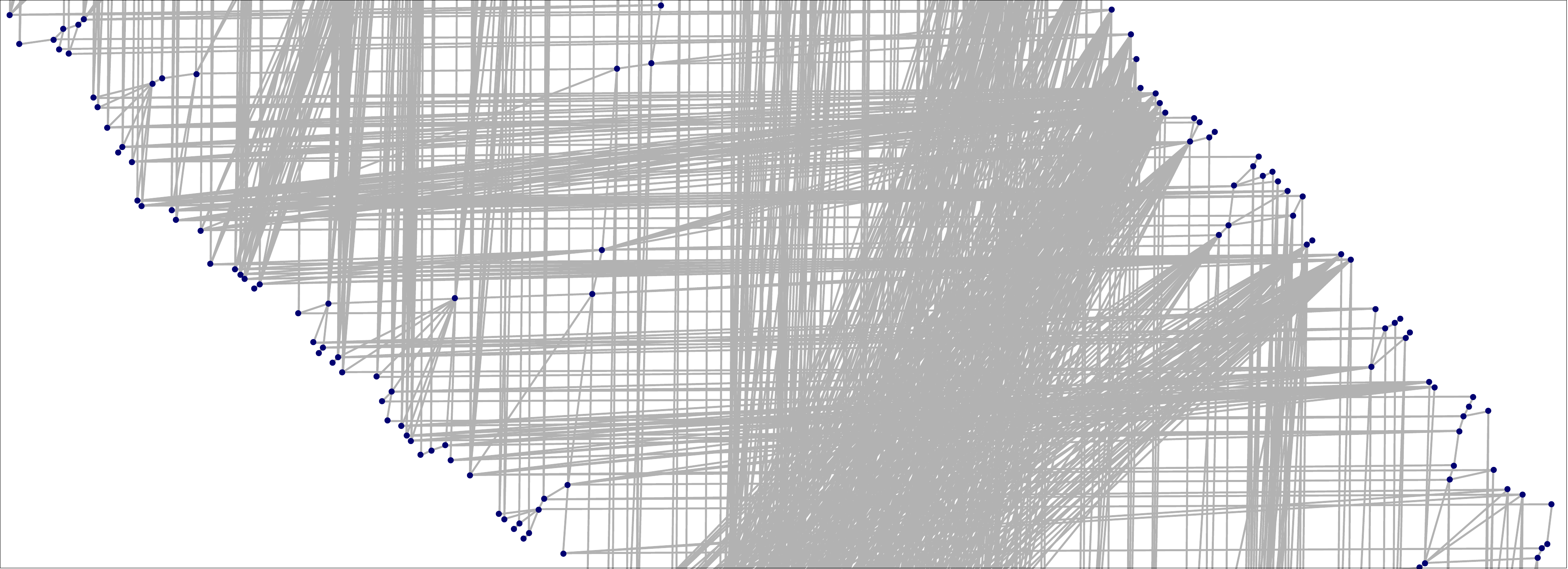}
\end{center}
\caption{The plot of a $\pdiamond$-avoider of length 1000 
and part of its Hasse graph}
\label{figEinar1000}
\end{figure}

\newpage
Under very general conditions, 
substructures of recursively defined combinatorial classes can be shown to be distributed normally in the limit.
By generalising the correlation polynomial of Guibas \& Odlyzko, 
and combining it with an application of the kernel method,
we prove that patterns in {\L}uka\-sie\-wicz paths also satisfy the conditions necessary for asymptotic normality:

\thmbox{
\begin{thm}\label{thmLukaPatternsGaussian}
The number of occurrences of a fixed pattern
in a {\L}uka\-sie\-wicz path of length $n$
exhibits
a Gaussian limit distribution
with mean
and standard deviation
asymptotically linear in $n$.
\end{thm}
}

In the next section, we introduce certain subsets of $\av(\pdiamond)$ for consideration, which consist of permutations having a particularly regular structure, and explore restrictions on their structure.
We follow this in Section~\ref{sectConcentration} by looking at a number of parameters that record the distribution of substructures in our permutations.
Key to our result is the fact that these are asymptotically concentrated, and
in this section
we
prove three of the four concentration results we need.
Section~\ref{sectLukasPatterns} is reserved for the proof of Theorem~\ref{thmLukaPatternsGaussian}, concerning the distribution of patterns in {\L}uka\-sie\-wicz paths. This section may be read independently of the rest of the paper.
To conclude,
in Section~\ref{sectLowerBound},
we use Theorem~\ref{thmLukaPatternsGaussian} to prove our final concentration result, and then
pull everything together to calculate a lower bound for $\gr(\av(\pdiamond))$, thus proving
Theorem~\ref{thm1324LowerBound}.

\section{Permutations with a regular structure}\label{sectW}

In this section, we present the structure and substructures of the permutations that we will be investigating.
Let $\WWW$ be the set of all
permutations avoiding $\pdiamond$
whose Hasse graphs are spanned by a sequence of trees
rooted alternately at the lower left and the upper right.
See Figure~\ref{fig157} for an example.

\begin{figure}[ht]
\begin{center}
  \includegraphics[scale=0.43]{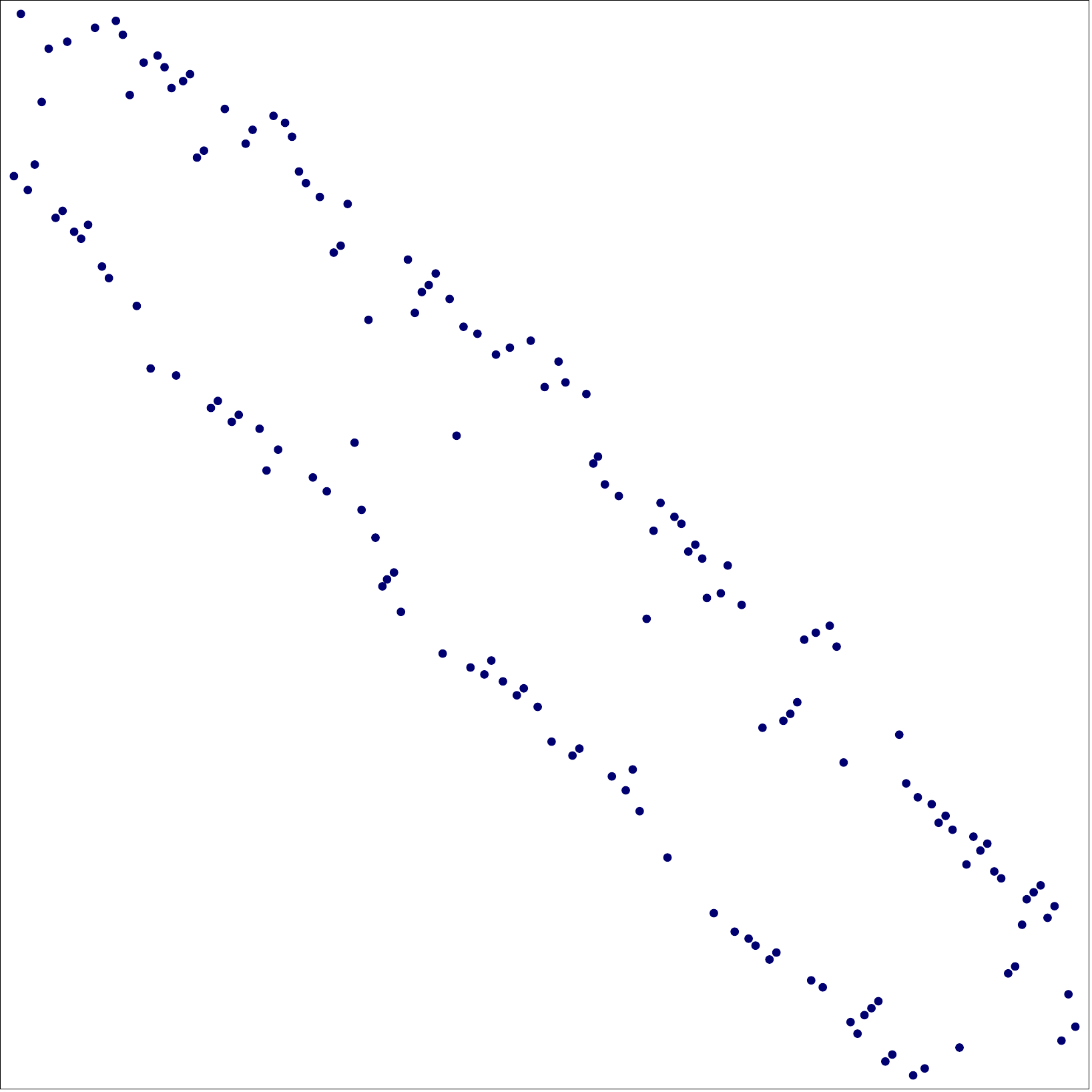}
  $
  \qquad
  $
  \includegraphics[scale=0.43]{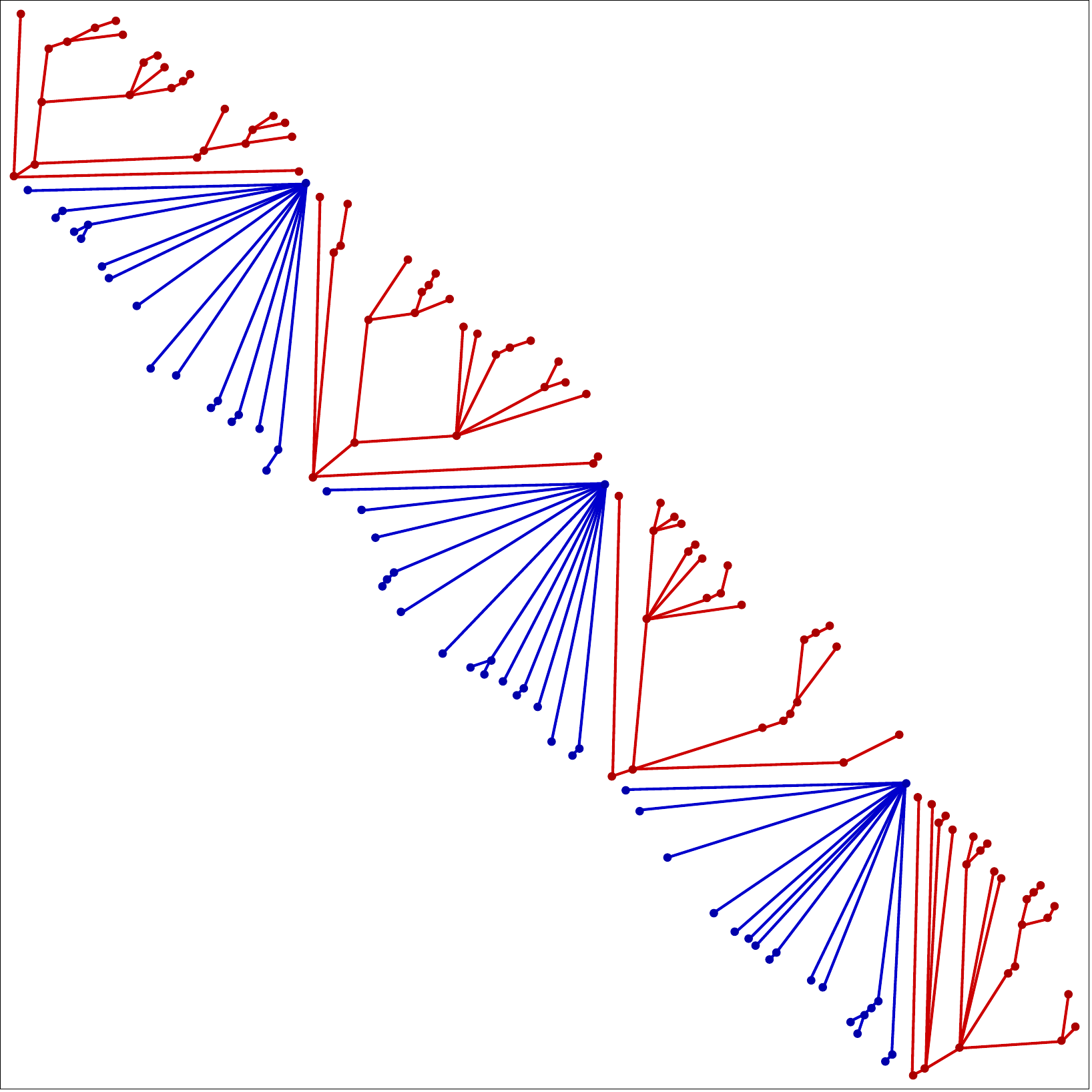}
\end{center}
  \caption{A permutation 
  in $\WWW(3,25,19,12)$ and the spanning of its Hasse graph by red and blue trees}
  \label{fig157}
\end{figure}
Trees rooted at a left-to-right minimum we colour red, and trees rooted at a right-to-left maximum we colour blue. We refer to these as \emph{red trees} and \emph{blue trees} respectively.
As a mnemonic, note that Red trees grow towards the Right and bLue trees grow towards the Left.

Observe that the root of each non-initial blue tree is the uppermost point below the root of the previous red tree, and the root of each non-initial red tree is the leftmost point to the right of the root of the previous blue tree.
Note that $\WWW$ does not contain every $\pdiamond$-avoider. For example, $\mathbf{2143}\notin\WWW$.

We consider elements of $\WWW$ with a particularly regular structure.
Each red tree will have the same number of vertices.
Similarly, each blue tree will have the same number of vertices. Moreover, every blue tree also will have the same root degree.
Specifically,
for any positive $t$, $k$, $\ell$ and~$d$, let $\WWW(t,k,\ell,d)$ be the
set of those permutations in $\WWW$ which
satisfy the following four conditions:
\begin{bulletnums}
  \item Its Hasse graph is spanned by $t+1$ red trees and $t$ blue trees.
  \item Each red tree has $k$ vertices.
  \item Each blue tree has $\ell$ vertices.
  \item Each blue tree has root degree $d$ .
\end{bulletnums}
See Figure~\ref{fig157} for an illustration of a permutation in $\WWW(3,25,19,12)$.

To simplify our presentation,
we will use the term \emph{blue subtree} to denote a principal subtree of a blue tree.
(The principal subtrees of a rooted tree are the connected components resulting from deleting the root.)
Thus each blue tree consists of a root vertex and a sequence of $d$ blue subtrees.
We will also refer to the roots of blue subtrees simply as \emph{blue roots}.

Our goal is to determine a lower bound for the growth rate of
the union of all the $\WWW(t,k,\ell,d)$.
To achieve this, our focus will be on sets in which the number and sizes of the trees
grow together along with
the root degree of blue trees.
Specifically, we consider
the parameterised sets
$$
\WWW_{\lambda,\delta}(k)
\;=\;
\WWW(k,k,\ceil{\lambda\+k},\ceil{\delta\+\lambda\+k}),
$$
for some $\lambda>0$ and $\delta\in(0,1)$,
consisting of $k+1$ $k$-vertex red trees and $k$ $\ceil{\lambda\+k}$-vertex blue trees each having root degree $\ceil{\delta\+\lambda\+k}$.
Thus,
$\lambda$ is the asymptotic ratio of the size of blue trees to red trees,
and $\delta$ is the limiting ratio of the root degree of each blue tree to its size.
Note that, asymptotically,
$1/\delta$
is the mean number of vertices in a blue subtree. Typically these subtrees will be small.

Let $g(\lambda,\delta)$ denote the upper growth rate
of $\bigcup_k\! \WWW_{\lambda,\delta}(k)$:
$$
g(\lambda,\delta)
\;=\;
\liminfty[k]\big|\WWW_{\lambda,\delta}(k)\big|^{\nfrac{1}{n(k,\lambda)}},
$$
where
$n(k,\lambda) = k\+\big(k+\ceil{\lambda\+k}+1\big)$ is the length of each permutation in $\WWW_{\lambda,\delta}(k)$.
In order to prove Theorem~\ref{thm1324LowerBound},
we will show that there is some
$\lambda$ and $\delta$ for which
$g(\lambda,\delta) > 9.81$.

\newpage
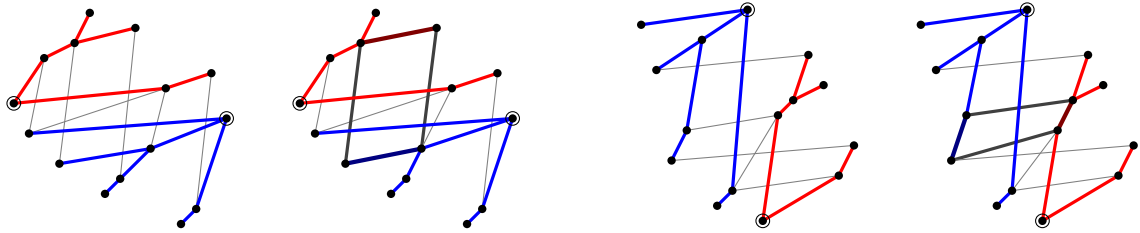
\begin{figure}[ht]
  $$
  \begin{tikzpicture}[scale=0.2,line join=round]
    \draw [black!50] (3,12)--(2,7)--(11,10);
    \draw [black!50] (5,13)--(4,5);
    \draw [black!50] (8,4)--(9,14);
    \draw [black!50] (10,6)--(11,10);
    \draw [black!50] (13,2)--(14,11);
    \draw [red,very thick] (6,15)--(5,13)--(3,12)--(1,9)--(11,10)--(14,11);
    \draw [red,very thick] (9,14)--(5,13);
    \draw [blue,very thick] (2,7)--(15,8)--(10,6)--(4,5);
    \draw [blue,very thick] (10,6)--(8,4)--(7,3);
    \draw [blue,very thick] (15,8)--(13,2)--(12,1);
    \plotpermnobox{15}{9,7,12,5,13,15,3,4,14,6,10,1,2,11,8}
    \draw [thin] (1,9) circle [radius=0.45];
    \draw [thin] (15,8) circle [radius=0.45];
  \end{tikzpicture}
  \qquad
  \begin{tikzpicture}[scale=0.2,line join=round]
    \draw [black!50] (3,12)--(2,7)--(11,10);
    \draw [black!75,very thick] (5,13)--(4,5);
    \draw [black!75,very thick] (9,6)--(10,14);
    \draw [black!50] (9,6)--(11,10);
    \draw [black!50] (13,2)--(14,11);
    \draw [red,very thick] (6,15)--(5,13)--(3,12)--(1,9)--(11,10)--(14,11);
    \draw [red!50!black,ultra thick] (10,14)--(5,13);
    \draw [blue,very thick] (2,7)--(15,8)--(9,6);
    \draw [blue!50!black,ultra thick] (9,6)--(4,5);
    \draw [blue,very thick] (9,6)--(8,4)--(7,3);
    \draw [blue,very thick] (15,8)--(13,2)--(12,1);
    \plotpermnobox{15}{9,7,12,5,13,15,3,4,6,14,10,1,2,11,8}
    \draw [thin] (1,9) circle [radius=0.45];
    \draw [thin] (15,8) circle [radius=0.45];
  \end{tikzpicture}
  \qquad\qquad
  \begin{tikzpicture}[scale=0.2,line join=round]
    \draw [black!50] (2,11)--(12,12);
    \draw [black!50] (4,7)--(10,8);
    \draw [black!50] (3,5)--(15,6);
    \draw [black!50] (10,8)--(7,3)--(14,4);
    \draw [blue,very thick] (1,14)--(8,15)--(5,13)--(2,11);
    \draw [blue,very thick] (5,13)--(4,7)--(3,5);
    \draw [blue,very thick] (8,15)--(7,3)--(6,2);
    \draw [red,very thick] (12,12)--(11,9)--(10,8)--(9,1)--(14,4)--(15,6);
    \draw [red,very thick] (13,10)--(11,9);
    \plotpermnobox{15}{14,11,5,7,13,2,3,15,1,8,9,12,10,4,6}
    \draw [thin] (9,1) circle [radius=0.45];
    \draw [thin] (8,15) circle [radius=0.45];
  \end{tikzpicture}
  \qquad
  \begin{tikzpicture}[scale=0.2,line join=round]
    \draw [black!50] (2,11)--(12,12);
    \draw [black!75,very thick] (3,5)--(10,7);
    \draw [black!75,very thick] (4,8)--(11,9);
    \draw [black!50] (3,5)--(15,6);
    \draw [black!50] (10,7)--(7,3)--(14,4);
    \draw [blue,very thick] (1,14)--(8,15)--(5,13)--(2,11);
    \draw [blue,very thick] (5,13)--(4,8);
    \draw [blue!50!black,ultra thick] (4,8)--(3,5);
    \draw [blue,very thick] (8,15)--(7,3)--(6,2);
    \draw [red,very thick] (12,12)--(11,9);
    \draw [red,very thick] (10,7)--(9,1)--(14,4)--(15,6);
    \draw [red!50!black,ultra thick] (11,9)--(10,7);
    \draw [red,very thick] (13,10)--(11,9);
    \plotpermnobox{15}{14,11,5,8,13,2,3,15,1,7,9,12,10,4,6}
    \draw [thin] (9,1) circle [radius=0.45];
    \draw [thin] (8,15) circle [radius=0.45];
  \end{tikzpicture}
  $$
  \caption{Valid and invalid horizontal interleavings,
            and valid and invalid vertical interleavings;
            occurrences of $\pdiamond$ are shown with thicker edges}
  \label{figInterleaveExample}
\end{figure}
$\WWW(t,k,\ell,d)$ consists precisely of those permutations that
can be built by starting with a $k$-vertex red tree
and repeating the following two steps exactly $t$ times (see Figure~\ref{fig157}):
\begin{bulletnums}
\item Place an $\ell$-vertex blue tree with root degree $d$
below the previous red tree (with its root to the right of the red tree),
horizontally interleaving its non-root vertices with the non-root vertices of the previous red tree in any way that avoids creating a $\pdiamond$.
\item Place a $k$-vertex red tree
to the right of the previous blue tree (with its root below the blue tree),
vertically interleaving its non-root vertices with the non-root vertices of the previous blue tree without creating a $\pdiamond$.
\end{bulletnums}
See Figure~\ref{figInterleaveExample} for illustrations of valid interleavings of the non-root vertices of red and blue trees, and also of invalid interleavings containing occurrences of $\pdiamond$.
The configurations that have to be avoided when interleaving are
shown schematically in Figure~\ref{fig1324Causes}.

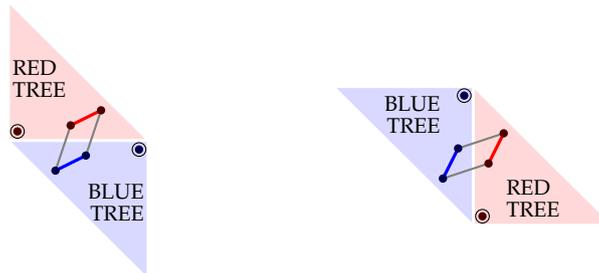
\begin{figure}[ht]
$$
\begin{tikzpicture}[line join=round,scale=0.2]
  \path [fill=blue!15] (7,-6) -- (7,2.9) -- (-1.9,2.9) -- (7,-6);
  \node[left] at (7.5,-0.4) {\scriptsize BLUE};
  \node[left] at (7.5,-1.85) {\scriptsize TREE};
  \path [fill=red!15] (6.9,3.1) -- (-2,3.1) -- (-2,12) -- (6.9,3.1);
  \node[right] at (-2.5,7.85) {\scriptsize RED};
  \node[right] at (-2.5,6.4) {\scriptsize TREE};
  \draw [red,very thick] (2,4)--(4,5);
  \draw [blue,very thick] (1,1)--(3,2);
  \draw [black!50, thick] (1,1)--(2,4);
  \draw [black!50, thick] (3,2)--(4,5);
  \plotpermnobox[blue!30!black]{}{1,0,2,0}
  \plotpermnobox[red!30!black]{}{0,4,0,5}
  \fill[blue!30!black,radius=0.275] (6.5,2.4) circle;
  \fill[red!30!black,radius=0.275] (-1.5,3.6) circle;
  \draw [thin] (6.5,2.4)  circle [radius=0.45];
  \draw [thin] (-1.5,3.6) circle [radius=0.45];
\end{tikzpicture}
\qquad\qquad\qquad
\raisebox{19pt}{
\begin{tikzpicture}[line join=round,scale=0.2]
  \path [fill=blue!15] (-6,7) -- (2.9,7) -- (2.9,-1.9) -- (-6,7);
  \node[left] at (1.5,6.0) {\scriptsize BLUE};
  \node[left] at (1.5,4.55) {\scriptsize TREE};
  \path [fill=red!15] (3.1,6.9) -- (3.1,-2) -- (12,-2) -- (3.1,6.9);
  \node[right] at (4.5,0.45) {\scriptsize RED};
  \node[right] at (4.5,-1.0) {\scriptsize TREE};
  \draw [red,very thick] (4,2)--(5,4);
  \draw [blue,very thick] (1,1)--(2,3);
  \draw [black!50, thick] (1,1)--(4,2);
  \draw [black!50, thick] (2,3)--(5,4);
  \plotpermnobox[blue!30!black]{}{1,3}
  \plotpermnobox[red!30!black]{}{0,0,0,2,4}
  \fill[blue!30!black,radius=0.275] (2.4,6.5) circle;
  \fill[red!30!black,radius=0.275] (3.6,-1.5) circle;
  \draw [thin] (2.4,6.5)  circle [radius=0.45];
  \draw [thin] (3.6,-1.5) circle [radius=0.45];
\end{tikzpicture}
}
$$
\caption{Possible causes of a $\pdiamond$ when interleaving horizontally and vertically}
\label{fig1324Causes}
\end{figure}

We will simply call a valid interleaving of the non-root vertices of a red tree with those of a blue tree an \emph{interleaving} of the trees.
Note that the choice of interleaving at each step is completely independent of the interleaving at any previous or subsequent
step. The only requirement is that no $\pdiamond$ is created by any of the interleavings.

The key to our result is thus an analysis of how vertices of red and blue trees may be interleaved without forming a $\pdiamond$. The remainder of the paper consists of this analysis.

In what follows, we will be working exclusively with interleavings of red and blue trees in elements of $\WWW_{\lambda,\delta}(k)$,
for some given
$\lambda>0$ and $\delta\in(0,1)$.
Thus, we will assume, without restatement, that a red tree has $k$ vertices, and that a blue tree has $\ell=\ceil{\lambda\+k}$ vertices and is composed of $d=\ceil{\delta\+\lambda\+k}$ blue subtrees.

We now consider how to avoid creating a $\pdiamond$.
Without loss of generality, we will limit our discussion to the horizontal case.

\begin{figure}[ht]
  $$
  \begin{tikzpicture}[scale=0.2,line join=round]
    \path [fill=blue!15] (5.65,0.5) rectangle (6.35,32.5);
    \path [fill=blue!15] (6.65,0.5) rectangle (8.35,32.5);
    \path [fill=blue!15] (10.65,0.5) rectangle (11.35,32.5);
    \path [fill=blue!15] (12.65,0.5) rectangle (15.35,32.5);
    \path [fill=blue!15] (18.65,0.5) rectangle (19.35,32.5);
    \path [fill=blue!15] (22.65,0.5) rectangle (26.35,32.5);
    \path [fill=blue!15] (26.65,0.5) rectangle (27.35,32.5);
    \path [fill=blue!15] (29.65,0.5) rectangle (31.35,32.5);
    \draw [gray,very thin] (6,15)--(9,30);
    \draw [gray,very thin] (6,15)--(10,29);
    \draw [gray,very thin] (6,15)--(12,27);
    \draw [gray,very thin] (6,15)--(16,24);
    \draw [gray,very thin] (6,15)--(20,21);
    \draw [gray,very thin] (6,15)--(22,18);
    \draw [gray,very thin] (6,15)--(29,19);
    \draw [gray,very thin] (8,14)--(9,30);
    \draw [gray,very thin] (8,14)--(10,29);
    \draw [gray,very thin] (8,14)--(12,27);
    \draw [gray,very thin] (8,14)--(16,24);
    \draw [gray,very thin] (8,14)--(20,21);
    \draw [gray,very thin] (8,14)--(22,18);
    \draw [gray,very thin] (8,14)--(29,19);
    \draw [gray,very thin] (11,12)--(12,27);
    \draw [gray,very thin] (11,12)--(16,24);
    \draw [gray,very thin] (11,12)--(20,21);
    \draw [gray,very thin] (11,12)--(22,18);
    \draw [gray,very thin] (11,12)--(29,19);
    \draw [gray,very thin] (15,11)--(16,24);
    \draw [gray,very thin] (15,11)--(20,21);
    \draw [gray,very thin] (15,11)--(22,18);
    \draw [gray,very thin] (15,11)--(29,19);
    \draw [gray,very thin] (20,21)--(19,8);
    \draw [gray,very thin] (22,18)--(19,8);
    \draw [gray,very thin] (29,19)--(19,8);
    \draw [gray,very thin] (26,7)--(28,20);
    \draw [gray,very thin] (26,7)--(29,19);
    \draw [gray,very thin] (27,3)--(28,20);
    \draw [gray,very thin] (27,3)--(29,19);
    \draw [gray,very thin] (6,15)--(18,23);
    \draw [gray,very thin] (8,14)--(18,23);
    \draw [gray,very thin] (11,12)--(18,23);
    \draw [gray,very thin] (15,11)--(18,23);
    \draw [red,very thick] (1,17)--(18,23);
    \draw [blue,very thick] (19,8)--(32,16);
    \fill [blue!80!white,thin] (6,15)  circle [radius=0.45];
    \fill [blue!80!white,thin] (8,14)  circle [radius=0.45];
    \fill [blue!80!white,thin] (11,12) circle [radius=0.45];
    \fill [blue!80!white,thin] (15,11) circle [radius=0.45];
    \fill [blue!80!white,thin] (19,8)  circle [radius=0.45];
    \fill [blue!80!white,thin] (26,7)  circle [radius=0.45];
    \fill [blue!80!white,thin] (27,3)  circle [radius=0.45];
    \fill [blue!80!white,thin] (31,2)  circle [radius=0.45];
    \draw [red,very thick] (1,17)--(2,31)--(3,32);
    \draw [red,very thick] (1,17)--(16,24)--(17,25);
    \draw [red,very thick] (1,17)--(4,26)--(5,28)--(9,30);
    \draw [red,very thick] (5,28)--(10,29);
    \draw [red,very thick] (4,26)--(12,27);
    \draw [red,very thick] (22,18)--(29,19);
    \draw [red,very thick] (1,17)--(20,21)--(21,22);
    \draw [red,very thick] (1,17)--(22,18)--(28,20);
    \draw [blue,very thick] (6,15)--(32,16);
    \draw [blue,very thick] (7,13)--(8,14)--(32,16);
    \draw [blue,very thick] (11,12)--(32,16);
    \draw [blue,very thick] (13,10)--(15,11)--(32,16);
    \draw [blue,very thick] (14,9)--(15,11);
    \draw [blue,very thick] (24,4)--(25,5)--(26,7)--(32,16);
    \draw [blue,very thick] (23,6)--(26,7);
    \draw [blue,very thick] (27,3)--(32,16);
    \draw [blue,very thick] (30,1)--(31,2)--(32,16);
    \draw [thin] (1,17)  circle [radius=0.4];
    \draw [thin] (32,16)  circle [radius=0.4];
    \plotpermnobox[red!30!black]{32}{17, 31, 32, 26, 28,  0,  0,  0, 30, 29,  0, 27,  0, 0,  0, 24, 25, 23, 0, 21, 22, 18, 0, 0, 0, 0, 0, 20, 19, 0, 0,  0}
    \plotpermnobox[blue!30!black]{32}{ 0,  0,  0,  0,  0, 15, 13, 14,  0,  0, 12,  0, 10, 9, 11,  0,  0,  0, 8,  0,  0,  0, 6, 4, 5, 7, 3, 0,  0, 1, 2, 16}
  \end{tikzpicture}
  $$
  \caption{An interleaving of red and blue trees in which no blue subtree (shown in a shaded rectangle) is split by a red vertex} 
  \label{figPartialInterleaving}
\end{figure}
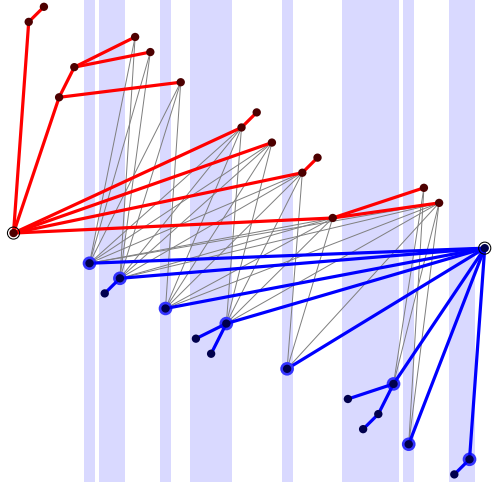
One way to guarantee that no $\pdiamond$ is created when interleaving trees is to ensure that no blue subtree is split by a red vertex, since the pattern $\pdiamond$ is avoided in any interleaving in which no red vertex occurs between two blue vertices of the same blue subtree. See Figure~\ref{figPartialInterleaving} for an illustration.

Let $\WWW^0_{\lambda,\delta}(k)$ be the subset of $\WWW_{\lambda,\delta}(k)$ in which red vertices are interleaved with blue subtrees in this manner in each interleaving.

$\WWW^0_{\lambda,\delta}(k)$ is easy to enumerate since trees and interleavings can be chosen independently. Indeed,
$$
\big| \WWW^0_{\lambda,\delta}(k) \big| \;=\; R_k^{\,k+1} \times B_k^{\,k} \times P_k^{\,2k},
$$
where $R_k$ is the number of distinct red trees, $B_k$ is the number of distinct blue trees and $P_k$ is the number of distinct 
ways of interleaving red vertices with blue subtrees.

$R_k=\frac{1}{k}\+\binom{2\+k-2}{k-1}$,
$B_k=\frac{d}{\ell-1}\+\binom{2\ell-3-d}{\ell-2}$ (see~\cite{FS2009}~Example~III.8),
and
$P_k = \binom{k-1 + d}{d}$.
Hence, by applying Stirling's approximation we obtain the following expression for the growth rate of $\WWW^0_{\lambda,\delta}(k)$:
\begin{equation}\label{eqW0GR}
g_0(\lambda,\delta)
\;=\;
\liminfty[k]\big|\WWW^0_{\lambda,\delta}(k)\big|^{\nfrac{1}{n(k,\lambda)}}
\;=\;
E(\lambda,\delta)^{\nfrac{1}{(1+\lambda)}},
\end{equation}
where
\begin{equation}\label{eqEDef}
  E(\lambda,\delta) \;=\;
  4 \+ \frac{(2-\delta)^{(2-\delta)\lambda}}{(1-\delta)^{(1-\delta)\lambda}} \+ \frac{(1+\delta\+\lambda)^{2(1+\delta\lambda)}}{(\delta\+\lambda)^{2\delta\lambda}} .
\end{equation}
It is now elementary to determine the maximum value of this growth rate.
For fixed $\lambda$, $E(\lambda,\delta)$ is maximal when $\delta$ has the value
$$
  \delta_\lambda \;=\; \frac{2\+ \lambda -1+\sqrt{1+4\+ \lambda +8\+ \lambda ^2}}{2\+ \lambda\+  (2+\lambda )}.
$$
Thence,
numerically maximising $g_0(\lambda,\delta)$ by setting $\lambda\approx0.61840$ (with $\delta_\lambda\approx0.86238$) yields a preliminary lower bound for  $\gr(\av(\pdiamond))$ of 9.40399.
It is rather a surprise that such a simple construction exhibits a growth rate as large as this.

Form this analysis, we see that we have complete freedom in choosing the positions of the blue roots (roots of blue subtrees) relative to the vertices of the red tree. In the light of this, we divide the process of interleaving into two stages:
\begin{bulletnums}
  \item Freely 
        interleave the blue roots with the red vertices.
  \item Select 
        positions for the non-root vertices of each blue subtree, while avoiding the creation of a $\pdiamond$.
\end{bulletnums}
We call the outcome of the first stage 
a \emph{pre-interleaving}.
A pre-interleaving is thus a sequence consisting of $k-1$ red vertices and $d=\ceil{\delta\+\lambda\+k}$ blue vertices (the blue roots);
the non-root vertices of the blue subtrees play no role in a pre-interleaving.

Note that in the second stage,
each blue subtree can be considered independently since no $\pdiamond$ can contain vertices from more than one blue subtree.
We now consider where the non-root vertices may be positioned.

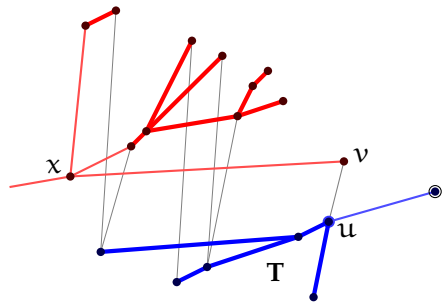
\begin{figure}[ht]
  $$
  \begin{tikzpicture}[scale=0.20,line join=round]
    \draw[gray,very thin] (3,5)--(4,21);
    \draw[gray,very thin] (3,5)--(5,12);
    \draw[gray,very thin] (8,3)--(9,19);
    \draw[gray,very thin] (10,4)--(11,18);
    \draw[gray,very thin] (10,4)--(12,14);
    \draw[gray,thin] (18,7)--(19,11);
    \draw[red!70!white,thick] (1,10)--(-3,9.3);
    \draw[red!70!white,thick] (1,10)--(2,20);
    \draw[red!70!white,thick] (1,10)--(5,12);
    \draw[red!70!white,thick] (1,10)--(19,11);
    \draw[red,ultra thick] (12,14)--(13,16);
    \draw[red,ultra thick] (12,14)--(15,15);
    \draw[red,ultra thick] (13,16)--(14,17);
    \draw[red,ultra thick] (2,20)--(4,21);
    \draw[red,ultra thick] (5,12)--(6,13);
    \draw[red,ultra thick] (6,13)--(9,19);
    \draw[red,ultra thick] (6,13)--(11,18);
    \draw[red,ultra thick] (6,13)--(12,14);
    \fill [blue!80!white,thin] (18,7)  circle [radius=0.4];
    \draw[blue!70!white,, thick] (18,7)--(25,9);
    \draw[blue,ultra thick] (3,5)--(16,6);
    \draw[blue,ultra thick] (8,3)--(10,4);
    \draw[blue,ultra thick] (10,4)--(16,6);
    \draw[blue,ultra thick] (16,6)--(18,7);
    \draw[blue,ultra thick] (17,2)--(18,7);
    \node[right] at (19,11.4) {$v$};
    \node[right] at (17.9,6.4) {$u$};
    \node[left] at (1.1,10.7) {$x$};
    \node at (14.5,3.7) {$\smallT$};
    \plotpermnobox[red!30!black]{}{10,20,0,21,12,13,0,0,19,0,18,14,16,17,15,0,0,0,11}
    \plotpermnobox[blue!30!black]{}{0,0,5,0,0,0,0,3,0,4,0,0,0,0,0,6,2,7,0,0,0,0,0,0,9}    
    \draw [thin] (25,9) circle [radius=0.4];
  \end{tikzpicture}
  $$
  \caption{An interleaving of a blue subtree $\smallT$ with its two-component red forest}
  \label{figRedForest}
\end{figure}
Our first observation is as follows:
Suppose $v$ is the nearest red vertex to the right of the root $u$ of some blue subtree $\smallT$.
Now let $x$ be the parent of $v$ in the red tree.
Then no vertex of~$\smallT$ can be positioned to the left of $x$, since otherwise a $\pdiamond$ would be created in which $xuv$ would be the $\mathbf{324}$.
Thus, vertices of~$\smallT$ can only be interleaved with those red vertices positioned between $u$ and $x$.
We will call the graph induced by this set of red vertices (which may be empty) a \emph{red forest}.
See Figure~\ref{figRedForest} for an illustration.

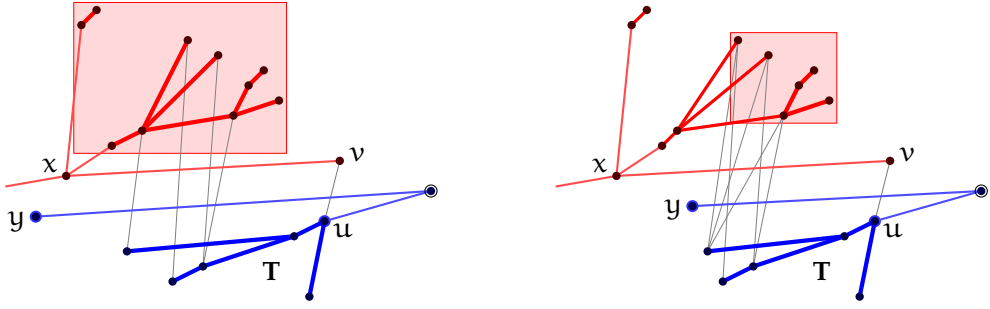
\begin{figure}
  $$
  \begin{tikzpicture}[scale=0.20,line join=round]
    \draw [red,very thin,fill=red!15] (1.5,11.5) rectangle (15.5,21.5);
    \draw[gray,very thin] (5,5)--(6,13);
    \draw[gray,very thin] (8,3)--(9,19);
    \draw[gray,very thin] (10,4)--(11,18);
    \draw[gray,very thin] (10,4)--(12,14);
    \draw[gray,thin] (18,7)--(19,11);
    \draw[red!70!white,thick] (1,10)--(-3,9.3);
    \draw[red!70!white,thick] (1,10)--(2,20);
    \draw[red!70!white,thick] (1,10)--(4,12);
    \draw[red!70!white,thick] (1,10)--(19,11);
    \draw[red,ultra thick] (12,14)--(13,16);
    \draw[red,ultra thick] (12,14)--(15,15);
    \draw[red,ultra thick] (13,16)--(14,17);
    \draw[red,ultra thick] (2,20)--(3,21);
    \draw[red,ultra thick] (4,12)--(6,13);
    \draw[red,ultra thick] (6,13)--(9,19);
    \draw[red,ultra thick] (6,13)--(11,18);
    \draw[red,ultra thick] (6,13)--(12,14);
    \draw[blue!70!white,, thick] (-1,7.3)--(25,9);
    \fill [blue!80!white,thin] (18,7)  circle [radius=0.4];
    \fill [blue!80!white,thin] (-1,7.3)  circle [radius=0.4];
    \draw[blue!70!white,, thick] (18,7)--(25,9);
    \draw[blue,ultra thick] (5,5)--(16,6);
    \draw[blue,ultra thick] (8,3)--(10,4);
    \draw[blue,ultra thick] (10,4)--(16,6);
    \draw[blue,ultra thick] (16,6)--(18,7);
    \draw[blue,ultra thick] (17,2)--(18,7);
    \node[right] at (19,11.4) {$v$};
    \node[right] at (17.9,6.4) {$u$};
    \node[left] at (1.1,10.7) {$x$};
    \node at (14.5,3.7) {$\smallT$};
    \plotpermnobox[red!30!black]{}{10,20,21,12,0,13,0,0,19,0,18,14,16,17,15,0,0,0,11}
    \plotpermnobox[blue!30!black]{}{0,0,0,0,5,0,0,3,0,4,0,0,0,0,0,6,2,7,0,0,0,0,0,0,9}
    \fill[blue!30!black,radius=0.275] (-1,7.3) circle;
    \node[left] at (-1,7.0) {$y$};
    \draw [thin] (25,9) circle [radius=0.4];
  \end{tikzpicture}
  \qquad\qquad
  \begin{tikzpicture}[scale=0.20,line join=round]
    \draw [red,very thin,fill=red!15] (8.5,13.5) rectangle (15.5,19.5);
    \draw[gray,very thin] (7,5)--(9,19);
    \draw[gray,very thin] (7,5)--(11,18);
    \draw[gray,very thin] (7,5)--(12,14);
    \draw[gray,very thin] (8,3)--(9,19);
    \draw[gray,very thin] (10,4)--(11,18);
    \draw[gray,very thin] (10,4)--(12,14);
    \draw[gray,thin] (18,7)--(19,11);
    \draw[red!70!white,thick] (1,10)--(-3,9.3);
    \draw[red!70!white,thick] (1,10)--(2,20);
    \draw[red!70!white,thick] (1,10)--(4,12);
    \draw[red!70!white,thick] (1,10)--(19,11);
    \draw[red,ultra thick] (12,14)--(13,16);
    \draw[red,ultra thick] (12,14)--(15,15);
    \draw[red,ultra thick] (13,16)--(14,17);
    \draw[red,very thick] (2,20)--(3,21);
    \draw[red,very thick] (4,12)--(5,13);
    \draw[red,very thick] (5,13)--(9,19);
    \draw[red,very thick] (5,13)--(11,18);
    \draw[red,very thick] (5,13)--(12,14);
    \draw[blue!70!white,, thick] (6,8)--(25,9);
    \fill [blue!80!white,thin] (18,7)  circle [radius=0.4];
    \fill [blue!80!white,thin] (6,8)  circle [radius=0.4];
    \draw[blue!70!white,, thick] (18,7)--(25,9);
    \draw[blue,ultra thick] (7,5)--(16,6);
    \draw[blue,ultra thick] (8,3)--(10,4);
    \draw[blue,ultra thick] (10,4)--(16,6);
    \draw[blue,ultra thick] (16,6)--(18,7);
    \draw[blue,ultra thick] (17,2)--(18,7);
    \node[right] at (19,11.4) {$v$};
    \node[right] at (17.9,6.4) {$u$};
    \node[left] at (1.1,10.7) {$x$};
    \node at (14.5,3.7) {$\smallT$};
    \plotpermnobox[red!30!black]{}{10,20,21,12,13,0,0,0,19,0,18,14,16,17,15,0,0,0,11}
    \plotpermnobox[blue!30!black]{}{0,0,0,0,0,0,5,3,0,4,0,0,0,0,0,6,2,7,0,0,0,0,0,0,9}
    \plotpermnobox[blue!30!black]{}{0,0,0,0,0,8}
    \node[left] at (6,7.8) {$y$};
    \draw [thin] (25,9) circle [radius=0.4];
  \end{tikzpicture}
  $$
\caption{Two interleavings of a blue subtree $\smallT$ with a red forest; the red fringes consist of the vertices in the shaded regions}
\label{figRedFringes}
\end{figure}
Our second (elementary) observation is 
is as follows:
Suppose $u$ is the root of some blue subtree $\smallT$, and $y$ is the next blue root to the left of $u$.
Then all the non-root vertices of $\smallT$ must occur to the right of $y$ (else $\smallT$ would not be a tree).
Note that $y$ may occur either to the left of $x$ or to its right.
See Figure~\ref{figRedFringes} for illustrations of both of these situations.

These two observations 
provide two independent constraints on the set of red vertices with which
the
non-root vertices of a blue subtree may be interleaved, the first determined by
the structure of the red tree
and the second by
the pre-interleaving.
This set consists of those vertices of the red forest situated to the right of both $x$ and $y$.
These red vertices induce a subgraph of the red forest which we will call its \emph{red fringe}.
In the examples in Figure~\ref{figRedFringes}, the red fringes consist of those vertices in the shaded regions.
The key fact that motivates the rest of our analysis is that vertices of a blue subtree may only be interleaved with vertices of its red fringe.

The size of a red fringe depends on both the size of the corresponding red forest and also on the location of the next blue root to the left.
Let us call the number of
red vertices positioned between a blue root~$u$ and the next blue root to its left ($y$) the \emph{gap size} of $u$; the gap size may be zero.
The number of vertices in the red fringe is thus the smaller of the gap size and the number of vertices in the red forest.

If we combine this fact
with results concerning the limiting distributions of
blue subtrees and
red fringes,
then we can establish a lower bound for $g(\lambda,\delta)$. This is the focus of the next section.

\section{Concentration of distributions}\label{sectConcentration}

To determine our lower bound, we depend critically on the fact that the asymptotic distributions of substructures of permutations in $\WWW_{\lambda,\delta}(k)$ are \emph{concentrated}. In this section we
introduce certain parameters counting these substructures,
show how their concentration enables us to bound $g(\lambda,\delta)$ from below,
and prove three of the four concentration results we require.

It is frequently the case that
distributions of parameters counting the proportion of particular substructures in combinatorial classes
have a convergent mean
and
a variance that
vanishes asymptotically.
As a direct consequence of Chebyshev's inequality, such distributions have the following concentration property (see~\cite{FS2009} Proposition~III.3): 
\begin{prop}\label{propConcentration}
If $\xi_n$ is a sequence of random variables with means $\mu_n=\mathbb{E}[\xi_n]$
and variances $\upsilon_n=\mathbb{V}[\xi_n]$ satisfying the conditions
$$
\liminfty\mu_n=\mu,
\qquad\qquad
\liminfty \upsilon_n \;=\; 0,
$$
for some constant $\mu$,
then
$\xi_n$ \emph{is concentrated at} $\mu$ in the sense that,
for any $\varepsilon>0$,
given sufficiently large $n$,
$$
\mathbb{P}\big[\, |\xi_n-\mu| \:\leqslant\: \varepsilon \,\big] \:\;>\;\: 1-\varepsilon.
$$
\end{prop}
In practice this often means that we can work on the assumption that the value of any such parameter is entirely concentrated at its limiting mean.
This is the case for the
parameters in which we are interested.

We will also make use of the following
result concerning \emph{multiple} concentrated parameters.
\begin{prop}\label{propConcentration2}
If $\xi_n$ and $\xi'_n$ are two sequences of random variables on the same sample space concentrated at $\mu$ and $\mu'$ respectively, then they are \emph{jointly concentrated}
in the sense that,
for any $\varepsilon>0$,
given sufficiently large~$n$,
$$
\mathbb{P}\big[\,
|\xi_n - \mu| \leqslant \varepsilon
\text{~~and~~}
|\xi'_n - \mu'| \leqslant \varepsilon
\,\big] \:\;>\:\; 1-\varepsilon.
$$
\end{prop}
\begin{proof}
For any $\eta>0$ and sufficiently large $n$,
the probability that $\xi_n$ differs from $\mu$ by less than $\eta$ exceeds $1-\eta$, and similarly for $\xi'$ with $\mu'$.
Hence the probability that both are simultaneously $\eta$-close to their asymptotic means is at least $1-2\+\eta$. Let $\eta=\nfrac{\varepsilon}{2}$.
\end{proof}

We now introduce the parameters we need:

\textbf{Blue subtrees} $\beta_k$:
For each plane tree $\smallT$, let $\beta_k(\smallT)$ be the random variable that records the
proportion of blue subtrees in a blue tree that are isomorphic to $\smallT$.

\textbf{Gap sizes} $\gamma_k$:
For each $j\geqslant0$,
let $\gamma_k(j)$ be the random variable that records the
proportion of blue roots in a pre-interleaving that have
gap size~$j$.
Also, let $\gamma_k(>\! j)$ record the
proportion of blue roots in a pre-interleaving whose gap size exceeds $j$.

\textbf{Red forests} $\rho_k$:
For each plane forest $\smallF$, let $\rho_k(\smallF)$ be the random variable that records the
proportion of positions in a red tree whose red forest is isomorphic to $\smallF$.
Also, let $\rho_k(\smallF^+)$ record the
proportion of positions in a red tree whose red forest has at least $|\smallF|$ vertices,
and for which the graph induced by the rightmost $|\smallF|$ vertices of the forest is isomorphic to $\smallF$.

Below, we prove that each of these parameters is concentrated, and calculate their asymptotic means.
First we describe how the parameters are combined.

Our first combined parameter counts red fringes.
Given
the combination of a red tree and a pre-interleaving of its vertices with a sequence of blue roots,
let $\varphi_k(\smallF)$
be the random variable that records the
proportion of blue roots whose red fringe is isomorphic to $\smallF$.
Now, occurrences of blue roots with a given gap size $j$ are spread almost uniformly across the positions in a red tree, non-uniformity only occurring for the $j$ leftmost positions.
This is also the case for the distribution of occurrences of blue roots whose gap size is at least $j$.
Hence, by the definition of a red fringe at the end of Section~\ref{sectW},
given any $\varepsilon>0$, if $k$ is large enough,
$\varphi_k(\smallF)$ differs from
\begin{equation}\label{eqChi}
\gamma_k(|\smallF|)\+\rho_k(\smallF^+)  \:+\:
\gamma_k(>\!|\smallF|)\+\rho_k(\smallF)
\end{equation}
by less than $\varepsilon$.

Our second combined parameter concerns pairs consisting of a blue subtree and a red fringe.
Given
a red tree, a blue tree and a pre-interleaving of their red vertices and blue roots,
let $\psi_k(\smallT,\smallF)$ be the random variable that records the proportion of blue subtrees that are isomorphic to $\smallT$ and have a red fringe that is isomorphic to~$\smallF$. We will call such a blue subtree a \emph{$(\smallT,\smallF)$-subtree}. Given that occurrences of a given blue subtree are distributed uniformly across the blue roots, we have
\begin{equation}\label{eqZeta}
\psi_k(\smallT,\smallF) \;=\;
\beta_k(\smallT)\+\varphi_k(\smallF).
\end{equation}

Since, as we show below,
$\beta_k$, $\gamma_k$ and $\rho_k$ are concentrated, it follows that $\psi_k$ is also concentrated.
Let $\mu(\smallT,\smallF)$ denote the limiting mean of $\psi_k(\smallT,\smallF)$
as $k$ tends to infinity.

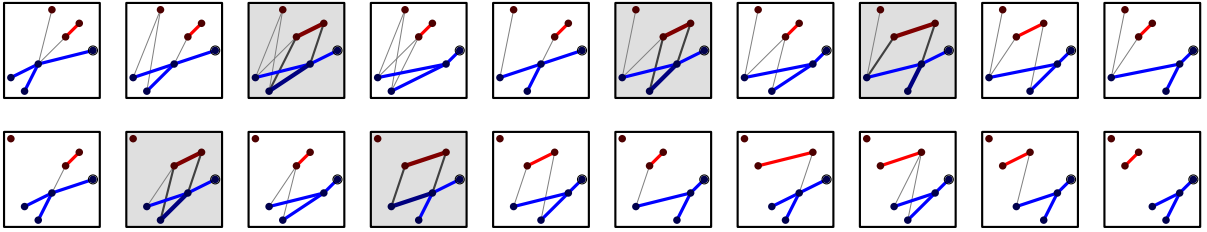
\begin{figure}[ht]
  $$
  \begin{tikzpicture}[scale=0.18,line join=round]
  \draw[gray,very thin] (3,3)--(4,7);
  \draw[gray,very thin] (3,3)--(5,5);
  \draw[red,very thick] (5,5)--(6,6);
  \draw[blue,very thick] (1,2)--(3,3);
  \draw[blue,very thick] (2,1)--(3,3);
  \draw[blue,very thick] (3,3)--(7,4);
  \plotpermnobox[red!30!black]{}{0,0,0,7,5,6}
  \plotpermnobox[blue!30!black]{}{2,1,3,0,0,0,4}
  \plotperm{7}{}
  \draw [thin] (7,4) circle [radius=0.33];
  \end{tikzpicture}
  \;\;\;
  \begin{tikzpicture}[scale=0.18,line join=round]
  \draw[gray,very thin] (1,2)--(3,7);
  \draw[gray,very thin] (2,1)--(3,7);
  \draw[gray,very thin] (4,3)--(5,5);
  \draw[red,very thick] (5,5)--(6,6);
  \draw[blue,very thick] (1,2)--(4,3);
  \draw[blue,very thick] (2,1)--(4,3);
  \draw[blue,very thick] (4,3)--(7,4);
  \plotpermnobox[red!30!black]{}{0,0,7,0,5,6}
  \plotpermnobox[blue!30!black]{}{2,1,0,3,0,0,4}
  \plotperm{7}{}
  \draw [thin] (7,4) circle [radius=0.33];
  \end{tikzpicture}
  \;\;\;
  \begin{tikzpicture}[scale=0.18,line join=round]
  \path [fill=gray!25] (0.5,0.5) rectangle (7.5,7.5);
  \draw[gray,very thin] (4,5)--(1,2)--(3,7);
  \draw[gray,very thin] (2,1)--(3,7);
  \draw[gray!50!black, thick] (4,5)--(2,1);
  \draw[gray!50!black, thick] (5,3)--(6,6);
  \draw[red!50!black,ultra thick] (4,5)--(6,6);
  \draw[blue,very thick] (1,2)--(5,3);
  \draw[blue!50!black,ultra thick] (2,1)--(5,3);
  \draw[blue,very thick] (5,3)--(7,4);
  \plotpermnobox[red!30!black]{}{0,0,7,5,0,6}
  \plotpermnobox[blue!30!black]{}{2,1,0,0,3,0,4}
  \plotperm{7}{}
  \draw [thin] (7,4) circle [radius=0.33];
  \end{tikzpicture}
  \;\;\;
  \begin{tikzpicture}[scale=0.18,line join=round]
  \draw[gray,very thin] (1,2)--(3,7);
  \draw[gray,very thin] (1,2)--(4,5);
  \draw[gray,very thin] (2,1)--(3,7);
  \draw[gray,very thin] (2,1)--(4,5);
  \draw[red,very thick] (4,5)--(5,6);
  \draw[blue,very thick] (1,2)--(6,3);
  \draw[blue,very thick] (2,1)--(6,3);
  \draw[blue,very thick] (6,3)--(7,4);
  \plotpermnobox[red!30!black]{}{0,0,7,5,6}
  \plotpermnobox[blue!30!black]{}{2,1,0,0,0,3,4}
  \plotperm{7}{}
  \draw [thin] (7,4) circle [radius=0.33];
  \end{tikzpicture}
  \;\;\;
  \begin{tikzpicture}[scale=0.18,line join=round]
  \draw[gray,very thin] (1,2)--(2,7);
  \draw[gray,very thin] (4,3)--(5,5);
  \draw[red,very thick] (5,5)--(6,6);
  \draw[blue,very thick] (1,2)--(4,3);
  \draw[blue,very thick] (3,1)--(4,3);
  \draw[blue,very thick] (4,3)--(7,4);
  \plotpermnobox[red!30!black]{}{0,7,0,0,5,6}
  \plotpermnobox[blue!30!black]{}{2,0,1,3,0,0,4}
  \plotperm{7}{}
  \draw [thin] (7,4) circle [radius=0.33];
  \end{tikzpicture}
  \;\;\;
  \begin{tikzpicture}[scale=0.18,line join=round]
  \path [fill=gray!25] (0.5,0.5) rectangle (7.5,7.5);
  \draw[gray,very thin] (4,5)--(1,2)--(2,7);
  \draw[gray!50!black, thick] (4,5)--(3,1);
  \draw[gray!50!black, thick] (5,3)--(6,6);
  \draw[red!50!black,ultra thick] (4,5)--(6,6);
  \draw[blue,very thick] (1,2)--(5,3);
  \draw[blue!50!black,ultra thick] (3,1)--(5,3);
  \draw[blue,very thick] (5,3)--(7,4);
  \plotpermnobox[red!30!black]{}{0,7,0,5,0,6}
  \plotpermnobox[blue!30!black]{}{2,0,1,0,3,0,4}
  \plotperm{7}{}
  \draw [thin] (7,4) circle [radius=0.33];
  \end{tikzpicture}
  \;\;\;
  \begin{tikzpicture}[scale=0.18,line join=round]
  \draw[gray,very thin] (1,2)--(2,7);
  \draw[gray,very thin] (1,2)--(4,5);
  \draw[gray,very thin] (3,1)--(4,5);
  \draw[red,very thick] (4,5)--(5,6);
  \draw[blue,very thick] (1,2)--(6,3);
  \draw[blue,very thick] (3,1)--(6,3);
  \draw[blue,very thick] (6,3)--(7,4);
  \plotpermnobox[red!30!black]{}{0,7,0,5,6}
  \plotpermnobox[blue!30!black]{}{2,0,1,0,0,3,4}
  \plotperm{7}{}
  \draw [thin] (7,4) circle [radius=0.33];
  \end{tikzpicture}
  \;\;\;
  \begin{tikzpicture}[scale=0.18,line join=round]
  \path [fill=gray!25] (0.5,0.5) rectangle (7.5,7.5);
  \draw[gray,very thin] (1,2)--(2,7);
  \draw[gray!50!black, thick] (3,5)--(1,2);
  \draw[gray!50!black, thick] (5,3)--(6,6);
  \draw[red!50!black,ultra thick] (3,5)--(6,6);
  \draw[blue,very thick] (1,2)--(5,3);
  \draw[blue!50!black,ultra thick] (4,1)--(5,3);
  \draw[blue,very thick] (5,3)--(7,4);
  \plotpermnobox[red!30!black]{}{0,7,5,0,0,6}
  \plotpermnobox[blue!30!black]{}{2,0,0,1,3,0,4}
  \plotperm{7}{}
  \draw [thin] (7,4) circle [radius=0.33];
  \end{tikzpicture}
  \;\;\;
  \begin{tikzpicture}[scale=0.18,line join=round]
  \draw[gray,very thin] (1,2)--(2,7);
  \draw[gray,very thin] (1,2)--(3,5);
  \draw[gray,very thin] (4,1)--(5,6);
  \draw[red,very thick] (3,5)--(5,6);
  \draw[blue,very thick] (1,2)--(6,3);
  \draw[blue,very thick] (4,1)--(6,3);
  \draw[blue,very thick] (6,3)--(7,4);
  \plotpermnobox[red!30!black]{}{0,7,5,0,6}
  \plotpermnobox[blue!30!black]{}{2,0,0,1,0,3,4}
  \plotperm{7}{}
  \draw [thin] (7,4) circle [radius=0.33];
  \end{tikzpicture}
  \;\;\;
  \begin{tikzpicture}[scale=0.18,line join=round]
  \draw[gray,very thin] (1,2)--(2,7);
  \draw[gray,very thin] (1,2)--(3,5);
  \draw[red,very thick] (3,5)--(4,6);
  \draw[blue,very thick] (1,2)--(6,3);
  \draw[blue,very thick] (5,1)--(6,3);
  \draw[blue,very thick] (6,3)--(7,4);
  \plotpermnobox[red!30!black]{}{0,7,5,6}
  \plotpermnobox[blue!30!black]{}{2,0,0,0,1,3,4}
  \plotperm{7}{}
  \draw [thin] (7,4) circle [radius=0.33];
  \end{tikzpicture}
  $$
  $$
  \begin{tikzpicture}[scale=0.18,line join=round]
  \draw[gray,very thin] (4,3)--(5,5);
  \draw[red,very thick] (5,5)--(6,6);
  \draw[blue,very thick] (2,2)--(4,3);
  \draw[blue,very thick] (3,1)--(4,3);
  \draw[blue,very thick] (4,3)--(7,4);
  \plotpermnobox[red!30!black]{}{7,0,0,0,5,6}
  \plotpermnobox[blue!30!black]{}{0,2,1,3,0,0,4}
  \plotperm{7}{}
  \draw [thin] (7,4) circle [radius=0.33];
  \end{tikzpicture}
  \;\;\;
  \begin{tikzpicture}[scale=0.18,line join=round]
  \path [fill=gray!25] (0.5,0.5) rectangle (7.5,7.5);
  \draw[gray,very thin] (4,5)--(2,2);
  \draw[gray!50!black, thick] (4,5)--(3,1);
  \draw[gray!50!black, thick] (5,3)--(6,6);
  \draw[red!50!black,ultra thick] (4,5)--(6,6);
  \draw[blue,very thick] (2,2)--(5,3);
  \draw[blue!50!black,ultra thick] (3,1)--(5,3);
  \draw[blue,very thick] (5,3)--(7,4);
  \plotpermnobox[red!30!black]{}{7,0,0,5,0,6}
  \plotpermnobox[blue!30!black]{}{0,2,1,0,3,0,4}
  \plotperm{7}{}
  \draw [thin] (7,4) circle [radius=0.33];
  \end{tikzpicture}
  \;\;\;
  \begin{tikzpicture}[scale=0.18,line join=round]
  \draw[gray,very thin] (2,2)--(4,5);
  \draw[gray,very thin] (3,1)--(4,5);
  \draw[red,very thick] (4,5)--(5,6);
  \draw[blue,very thick] (2,2)--(6,3);
  \draw[blue,very thick] (3,1)--(6,3);
  \draw[blue,very thick] (6,3)--(7,4);
  \plotpermnobox[red!30!black]{}{7,0,0,5,6}
  \plotpermnobox[blue!30!black]{}{0,2,1,0,0,3,4}
  \plotperm{7}{}
  \draw [thin] (7,4) circle [radius=0.33];
  \end{tikzpicture}
  \;\;\;
  \begin{tikzpicture}[scale=0.18,line join=round]
  \path [fill=gray!25] (0.5,0.5) rectangle (7.5,7.5);
  \draw[gray!50!black, thick] (2,2)--(3,5);
  \draw[gray!50!black, thick] (5,3)--(6,6);
  \draw[red!50!black,ultra thick] (3,5)--(6,6);
  \draw[blue!50!black,ultra thick] (2,2)--(5,3);
  \draw[blue,very thick] (4,1)--(5,3);
  \draw[blue,very thick] (5,3)--(7,4);
  \plotpermnobox[red!30!black]{}{7,0,5,0,0,6}
  \plotpermnobox[blue!30!black]{}{0,2,0,1,3,0,4}
  \plotperm{7}{}
  \draw [thin] (7,4) circle [radius=0.33];
  \end{tikzpicture}
  \;\;\;
  \begin{tikzpicture}[scale=0.18,line join=round]
  \draw[gray,very thin] (2,2)--(3,5);
  \draw[gray,very thin] (4,1)--(5,6);
  \draw[red,very thick] (3,5)--(5,6);
  \draw[blue,very thick] (2,2)--(6,3);
  \draw[blue,very thick] (4,1)--(6,3);
  \draw[blue,very thick] (6,3)--(7,4);
  \plotpermnobox[red!30!black]{}{7,0,5,0,6}
  \plotpermnobox[blue!30!black]{}{0,2,0,1,0,3,4}
  \plotperm{7}{}
  \draw [thin] (7,4) circle [radius=0.33];
  \end{tikzpicture}
  \;\;\;
  \begin{tikzpicture}[scale=0.18,line join=round]
  \draw[gray,very thin] (2,2)--(3,5);
  \draw[red,very thick] (3,5)--(4,6);
  \draw[blue,very thick] (2,2)--(6,3);
  \draw[blue,very thick] (5,1)--(6,3);
  \draw[blue,very thick] (6,3)--(7,4);
  \plotpermnobox[red!30!black]{}{7,0,5,6}
  \plotpermnobox[blue!30!black]{}{0,2,0,0,1,3,4}
  \plotperm{7}{}
  \draw [thin] (7,4) circle [radius=0.33];
  \end{tikzpicture}
  \;\;\;
  \begin{tikzpicture}[scale=0.18,line join=round]
  \draw[gray,very thin] (5,3)--(6,6);
  \draw[red,very thick] (2,5)--(6,6);
  \draw[blue,very thick] (3,2)--(5,3);
  \draw[blue,very thick] (4,1)--(5,3);
  \draw[blue,very thick] (5,3)--(7,4);
  \plotpermnobox[red!30!black]{}{7,5,0,0,0,6}
  \plotpermnobox[blue!30!black]{}{0,0,2,1,3,0,4}
  \plotperm{7}{}
  \draw [thin] (7,4) circle [radius=0.33];
  \end{tikzpicture}
  \;\;\;
  \begin{tikzpicture}[scale=0.18,line join=round]
  \draw[gray,very thin] (3,2)--(5,6);
  \draw[gray,very thin] (4,1)--(5,6);
  \draw[red,very thick] (2,5)--(5,6);
  \draw[blue,very thick] (3,2)--(6,3);
  \draw[blue,very thick] (4,1)--(6,3);
  \draw[blue,very thick] (6,3)--(7,4);
  \plotpermnobox[red!30!black]{}{7,5,0,0,6}
  \plotpermnobox[blue!30!black]{}{0,0,2,1,0,3,4}
  \plotperm{7}{}
  \draw [thin] (7,4) circle [radius=0.33];
  \end{tikzpicture}
  \;\;\;
  \begin{tikzpicture}[scale=0.18,line join=round]
  \draw[gray,very thin] (3,2)--(4,6);
  \draw[red,very thick] (2,5)--(4,6);
  \draw[blue,very thick] (3,2)--(6,3);
  \draw[blue,very thick] (5,1)--(6,3);
  \draw[blue,very thick] (6,3)--(7,4);
  \plotpermnobox[red!30!black]{}{7,5,0,6}
  \plotpermnobox[blue!30!black]{}{0,0,2,0,1,3,4}
  \plotperm{7}{}
  \draw [thin] (7,4) circle [radius=0.33];
  \end{tikzpicture}
  \;\;\;
  \begin{tikzpicture}[scale=0.18,line join=round]
  \draw[red,very thick] (2,5)--(3,6);
  \draw[blue,very thick] (4,2)--(6,3);
  \draw[blue,very thick] (5,1)--(6,3);
  \draw[blue,very thick] (6,3)--(7,4);
  \plotpermnobox[red!30!black]{}{7,5,6}
  \plotpermnobox[blue!30!black]{}{0,0,0,2,1,3,4}
  \plotperm{7}{}
  \draw [thin] (7,4) circle [radius=0.33];
  \end{tikzpicture}
  $$
  \caption{
  $Q(\mathbf{2134},\mathbf{312})=15$; the five shaded interleavings contain a $\pdiamond$
  }
  \label{figTauPiInterleaving}
\end{figure}
Finally, given a blue subtree~$\smallT$ and a red fringe $\smallF$, let
$Q(\smallT,\smallF)$ denote the number of distinct ways of interleaving the non-root vertices of $\smallT$ and the vertices of $\smallF$ without creating a $\pdiamond$. 
See Figure~\ref{figTauPiInterleaving} for an example.

With all the relevant parameters defined, we are now in a position to present a lower bound on the value of $g(\lambda,\delta)$.
\begin{prop}\label{propLB}
Let $\SSS$ be any finite set of pairs $(\smallT,\smallF)$ composed of a plane tree $\smallT$ and a plane forest $\smallF$. Then
$$
g(\lambda,\delta)
\:\;\geqslant\;\:
E(\lambda,\delta)
^{\nfrac{1}{(1+\lambda)}}
\,\times\!
\prod_{(\subT,\subF)\,\in\,\SSS\,}\!\!\!
Q(\smallT,\smallF)^{2\delta\lambda\mu(\subT,\subF)
/(1+\lambda)
},
$$
where $E(\lambda,\delta)$ is as defined in~\eqref{eqEDef} on page~\pageref{eqEDef}.
\end{prop}
\begin{proof}
Consider
a red tree and a blue tree together with a pre-interleaving of their red vertices and blue roots.
By Propositions~\ref{propConcentration} and~\ref{propConcentration2},
for any $\varepsilon>0$, if $k$ is large enough, then
with probability exceeding $1-\varepsilon$,
it is the case that
$\big|\psi_k(\smallT,\smallF) - \mu(\smallT,\smallF)\big| \leqslant \varepsilon$
for every $(\smallT,\smallF)\in \SSS$.

So the proportion of pre-interleaved pairs of trees
with at least
$\ceil{\delta\+\lambda\+k}\!(\mu(\smallT,\smallF)-\varepsilon)$
occurrences of
$(\smallT,\smallF)$-subtrees
for every $(\smallT,\smallF)\in \SSS$
exceeds $1-\varepsilon$.

Elements of $\WWW_{\lambda,\delta}(k)$
are constructed by independently choosing
trees and interleavings.
Thus, the size of $\WWW_{\lambda,\delta}(k)$ is bounded below by
$$
\big|\WWW_{\lambda,\delta}(k)\big|
\:\;\geqslant\;\:
\big| \WWW^0_{\lambda,\delta}(k) \big| \,\times\,
\Bigg(\!
\prod_{(\subT,\subF)\,\in\,\SSS\,}\!\!\!
(1-\varepsilon)\+Q(\smallT,\smallF)^{\ceil{\delta\lambda k}(\mu(\subT,\subF)-\varepsilon)}
\Bigg) ^{\!2k}
.$$

Recall that
$$
g(\lambda,\delta)
\;=\;
\liminfty[k]\big|\WWW_{\lambda,\delta}(k)\big|^{\nfrac{1}{n(k,\lambda)}},
$$
where
$n(k,\lambda) = k\+\big(k+\ceil{\lambda\+k}+1\big)$ is the length of each permutation in $\WWW_{\lambda,\delta}(k)$.
The desired result follows after expanding and taking the limit, making use of~\eqref{eqW0GR}.
\end{proof}

To determine the asymptotic mean and variance of our parameters, we utilise bivariate generating functions.
The following standard result enables us to obtain the required moments directly
as long as we can extract coefficients.
We use $[z^n]\+f(z)$ to denote the coefficient of $z^n$ in the series expansion of $f(z)$; we also use $f_x$ for $\frac{\partial f}{\partial x}$ and $f_{xx}$ for $\frac{\partial^2 f}{\partial x^2}$.
\begin{prop}[\cite{FS2009} Proposition III.2]\label{propMoments}
Suppose $A(z,x)$ is the bivariate generating function for some combinatorial class, in which $z$ marks size and $x$ marks the value of a parameter $\xi$.
Then the mean
and
variance
of $\xi$ for elements of size $n$ are given by
$$
\mathbb{E}_n[\xi] \;=\; \frac{[z^n]\+ A_x(z,1)}{[z^n]\+ A(z,1)}
\qquad
\text{and}
\qquad
\mathbb{V}_n[\xi]
\;=\;
\frac{[z^n]\+ A_{xx}(z,1) 
}{[z^n]\+ A(z,1)}
\:+\: \mathbb{E}_n(\xi)
\:-\: {\mathbb{E}_n(\xi)}^2
$$
respectively. 
\end{prop}

The proofs of our first three concentration results each follow a similar pattern:
establish the generating function; extract the coefficients; apply Proposition~\ref{propMoments}; take limits
using Stirling's approximation; finally apply Proposition~\ref{propConcentration}.

First, we consider blue subtrees.
Recall that the random variable $\beta_k(\smallT)$ records
the proportion of principal subtrees in a
$\ceil{\lambda\+k}$-vertex plane tree with root degree $\ceil{\delta\+\lambda\+k}$
that are isomorphic to~$\smallT$.
\begin{prop}\label{propPsi}
Let $i=|\smallT|$.
$\beta_k(\smallT)$ is concentrated at
$$
\mu_\beta(\smallT) \;=\; \frac{(1-\delta)^{i-1}}{(2-\delta)^{2i-1}}.
$$
\end{prop}
\begin{proof}
Let
$
T(z) = \thalf \+ \big(1-\sqrt{1-4\+z}\big)
$
be the generating function for plane trees.
Then
the bivariate generating function for plane trees with root degree $d$, in which $z$ marks vertices and $u$ marks principal subtrees isomorphic to $\smallT$, is given by
$$
B(z,u) \;=\; z \left(T(z)+(u-1)\+z^i\right)^{\!d}.
$$
Extracting coefficients yields
$$
\begin{array}{lcl}
[z^\ell]\+ B(z,1) & = & \frac{d}{\ell-1} \+ \binom{2\+\ell - d - 3}{\ell-2}, \\[6pt]
[z^\ell]\+ B_u(z,1)
& = & \frac{d\+(d-1)}{\ell-i-1}\+\binom{2\+\ell-2\+i - d - 2}{\ell-i-2}, \\[6pt]
[z^\ell]\+ B_{uu}(z,1)
& = & \frac{d\+(d-1)\+(d-2)}{\ell-2\+i-1}\+\binom{2\+\ell-4\+i - d - 1}{\ell-2\+i-2}.
\end{array}
$$
Hence, with $\ell=\ceil{\lambda\+k}$ and $d=\ceil{\delta\+\lambda\+k}$, applying Proposition~\ref{propMoments}
and taking limits gives
$$\liminfty[k]\mathbb{E}[\beta_k(\smallT)] \;=\; \frac{(1-\delta)^{i-1}}{(2-\delta)^{2i-1}}
\qquad\text{and}\qquad
\liminfty[k] k \+ \mathbb{V}[\beta_k(\smallT)] \;=\; \upsilon_\beta(\smallT),$$
where $\upsilon_\beta(\smallT)$ is some rational function in $\delta$.
So,
$\beta_k(\smallT)$ satisfies the conditions for Proposition~\ref{propConcentration} and is thus concentrated at $\mu_\beta(\smallT)$ as required.
\end{proof}
Secondly, we consider gap size.
Recall that,
given a pre-interleaving of the non-root vertices of a $k$-vertex red tree and
$\ceil{\delta\+\lambda\+k}$ blue roots,
the random variable
$\gamma_k(j)$ records the
proportion of blue roots that have
gap size~$j$. Similarly, $\gamma_k(>\! j)$ records the proportion that have gap size exceeding $j$.
\begin{prop}\label{propPhi}
$\gamma_k(j)$ is concentrated at
$$\mu_\gamma(j) \;=\;
\frac{\delta\+\lambda}{(1+\delta\+\lambda)^{j+1}}
.
$$
Also,
$\gamma_k(>\! j)$ is concentrated at
$$\mu_\gamma(>\!j) \;=\;
\frac{1}{(1+\delta\+\lambda)^{j+1}}
.
$$
\end{prop}
\begin{proof}
The bivariate generating function for pre-interleavings containing $d$ blue roots, in which $z$ marks red vertices and $v$ marks gaps of size $j$, is given by
$$
G(z,v) \;=\; \tfrac{z}{1-z} \left(\tfrac{1}{1-z}+(v-1)\+z^j\right)^{\!d}.
$$
Extracting coefficients yields
$$
\begin{array}{lcl}
[z^k]\+ G(z,1) & = & \binom{k+d-1}{d}, \\[6pt]
[z^k]\+ G_v(z,1)
& = & d\+\binom{k-j+d-2}{d-1}, \\[6pt]
[z^k]\+ G_{vv}(z,1)
& = & d\+(d-1)\+\binom{k-2\+j+d-3}{d-2}.
\end{array}
$$
Hence, with $d=\ceil{\delta\+\lambda\+k}$, applying Proposition~\ref{propMoments}
and taking limits gives
$$\liminfty[k]\mathbb{E}[\gamma_k(j)] \;=\; \frac{\delta\+\lambda}{(1+\delta\+\lambda)^{j+1}}
\qquad\text{and}\qquad
\liminfty[k] k \+ \mathbb{V}[\gamma_k(j)] \;=\; \upsilon_\gamma(j),$$
where $\upsilon_\gamma(j)$ is some rational function in $\delta$ and $\lambda$.
So,
$\gamma_k(j)$ satisfies the conditions for Proposition~\ref{propConcentration} and is thus concentrated at $\mu_\gamma(j)$ as required.

Also, since
$$
\liminfty[k]\mathbb{E}[\gamma_k(>\!j)] \;=\; 1\:-\:\sum\limits_{i=0}^j\mu_\gamma(i) \;=\; \frac{1}{(1+\delta\+\lambda)^{j+1}}
,
$$ $\gamma_k(>\!j)$ is concentrated at $\mu_\gamma(>\!j)$ as required.
\end{proof}
Thirdly, we consider red forests.
Recall that
the random variable
$\rho_k(\smallF)$ records the
proportion of positions in a $k$-vertex red tree whose red forest is isomorphic to $\smallF$.
\begin{prop}\label{propOmega}
Let $m=|\smallF|$.
$\rho_k(\smallF)$ is concentrated at
$$
\mu_\rho(\smallF)
\;=\;
\frac{1}{2^{2m+1}}.
$$
\end{prop}
\begin{proof}
If $\smallF$ has $h$ components, then an occurrence of $\smallF$ in a red tree
comprises the leftmost $h$ subtrees of some vertex $x$ that has at least one additional child vertex to the right.
See Figure~\ref{figRedFringes} for an illustration.
Hence, if $\RRR$ is the class of red trees augmented by marking occurrences of $\smallF$ with~$w$, then $\RRR$ satisfies the structural equation
$$
\RRR \;=\; z \left( \seq{\RRR} \:+\: (w-1) \+ z^m \+ \seqplus{\RRR} \right) .
$$
So the corresponding bivariate generating function, $R(z,w)$, satisfies the functional equation
$$
R(z,w) \;=\; \frac{z \left( 1 \:+\: (w-1) \+ z^m \+ R(z,w) \right)}{1-R(z,w)} ,
$$
and hence
$$
R(z,w) \;=\; \thalf \+ \Big( 1 + (1-w)\+z^{m+1} - \sqrt{\big( 1 + (1-w)\+z^{m+1} \big)^2 - 4\+z} \,\Big) .
$$
Extracting coefficients then yields
$$
\begin{array}{lcl}
[z^k]\+ R(z,1) & = & \frac{1}{k} \binom{2k-2}{k-1}, \\[6pt]
[z^k]\+ R_w(z,1)
& = & \binom{2k-2m-3}{k-m-1}, \\[6pt]
[z^k]\+ R_{ww}(z,1)
& = & (k-2m-2)\+ \binom{2k-4m-4}{k-2m-2}.
\end{array}
$$
Hence, applying Proposition~\ref{propMoments}
and taking limits gives
$$\liminfty[k]\mathbb{E}[\rho_k(\smallF)] \;=\; \frac{1}{2^{2\+m+1}}
\qquad\text{and}\qquad
\liminfty[k] k \+ \mathbb{V}[\rho_k(\smallF)] \;=\; \upsilon_\rho(\smallF),$$
where $\upsilon_\rho(\smallF)$ depends only on $|\smallF|$.
So,
$\rho_k(\smallF)$ satisfies the conditions for Proposition~\ref{propConcentration} and is thus concentrated at $\mu_\rho(\smallF)$ as required.
\end{proof}

\begin{figure}[ht]
  $$
  \begin{tikzpicture}[scale=0.24,line join=round]
    \draw [thin,fill=gray!30!white] (9.6,12.4) to[out=45,in=166](10,12.5657) to[out=346,in=121](12.4,10.4) to[out=301,in=45](12.4,9.6)--(11.4,8.6) to[out=225,in=301](10.6,8.6) to[out=121,in=270](9.4343,12) to[out=90,in=225](9.6,12.4);
    \draw [thin,fill=gray!30!white] (7.6,13.4) to[out=45,in=180](8,13.5657) to[out=0,in=135](10.4,12.4) to[out=315,in=45](10.4,11.6)--(9.4,10.6) to[out=225,in=315](8.6,10.6) to[out=135,in=270](7.4343,13) to[out=90,in=225](7.6,13.4);
    \draw [thin] (9.6,12.4) to[out=45,in=166](10,12.5657) to[out=346,in=121](12.4,10.4) to[out=301,in=45](12.4,9.6)--(11.4,8.6) to[out=225,in=301](10.6,8.6) to[out=121,in=270](9.4343,12) to[out=90,in=225](9.6,12.4);
    \draw [thin,fill=gray!30!white] (1.6,18.4) to[out=45,in=180](2,18.5657) to[out=0,in=135](4.4,17.4) to[out=315,in=45](4.4,16.6)--(3.4,15.6) to[out=225,in=315](2.6,15.6) to[out=135,in=270](1.4343,18) to[out=90,in=225](1.6,18.4);
    \draw [ultra thick] (3,16)--(4,17);
    \draw [ultra thick] (9,11)--(10,12);
    \draw [ultra thick] (11,9)--(12,10);
    \draw [red,very thick] (2,18)--(1,15)--(3,16);
    \draw [red,very thick] (7,14)--(6,8)--(8,13);
    \draw [red,very thick] (9,11)--(6,8)--(5,1)--(11,9);
    \draw [red,very thick] (15,7)--(14,6)--(13,3)--(16,5);
    \draw [red,very thick] (17,4)--(13,3)--(5,1)--(18,2);
    \plotpermnobox[red!30!black]{18}{15,0,0,0,1,8,14,0,0,0,0,0,3,6,7,5,4,2}
    \plotpermnobox{18}{0,18,16,17,0,0,0,13,11,12,9,10}
  \end{tikzpicture}
  \qquad\qquad
  \raisebox{36pt}{
  \begin{tikzpicture}[scale=0.33,line join=round]
    \draw[help lines] (0,0) grid (18,6);
    \draw[red,thick] (0,0)--(4,4)--(5,4)--(6,2);
    \draw[red,thick] (11,5)--(12,6)--(13,3)--(14,1);
    \draw[red,thick] (17,3)--(18,2);
    \draw[ultra thick] (6,2)--(7,3)--(8,3)--(9,4)--(10,4)--(11,5);
    \draw[ultra thick] (14,1)--(15,2)--(16,2)--(17,3);
  \end{tikzpicture}
  }
  $$
  \caption{A partial red tree and the corresponding {\L}uka\-sie\-wicz path; 
  the three marked red fringes correspond to the occurrences of the pattern $\mathbf{1,0,1}$}
  \label{figRedTreePatterns}
\end{figure}
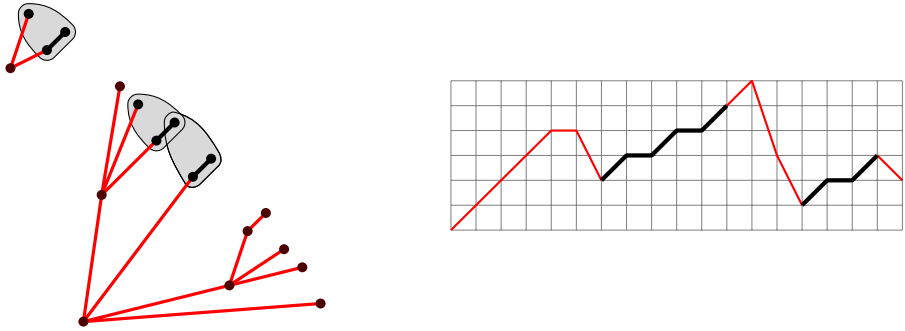
Our fourth and final concentration result concerns red fringes.
Recall that the random variable
$\rho_k(\smallF^+)$ records the
proportion of positions in a red tree whose red forest has at least $|\smallF|$ vertices,
and for which the graph induced by the rightmost $|\smallF|$ vertices of the forest is isomorphic to $\smallF$.

We would like to determine the bivariate generating function for red trees in which occurrences of the red fringe $\smallF$ are marked.
This is considerably less straightforward than was the case for the other parameters.
Primarily, this is because distinct occurrences of $\smallF$ may overlap. See the left of Figure~\ref{figRedTreePatterns} for an illustration.
To achieve our goal, it is convenient to rephrase our problem in terms of {\L}uka\-sie\-wicz paths.

Recall from Section~\ref{sectIntro} that a {\L}uka\-sie\-wicz path of length $n$ is a
sequence of integers $y_0,\ldots,y_n$ such that $y_0=0$,
$y_i\geqslant1$ for $i\geqslant1$,
and each step $s_i=y_i-y_{i-1}\leqslant1$.
It is easy to see that {\L}uka\-sie\-wicz paths
are in bijection with
red trees: visit the vertices of the tree from \emph{right to left} and let the height of the path be equal to the number of components in the forest induced by the vertices visited so far.
Thus, for each leaf vertex, the path contains an up-step, and for each internal vertex with $r$ children, the path contains a $(1\!-\!r)$-step.
See Figure~\ref{figRedTreePatterns} for an illustration.

Recall also that a pattern $\omega$ of length $m$ in a {\L}uka\-sie\-wicz path is a sequence of contiguous steps $\omega_1,\ldots,\omega_m$ in the path such that $\sum_{j=1}^i \omega_j > 0$
for $1\leqslant i\leqslant m$.
We do not consider sequences of steps for which the height drops to zero or below.
Thus, a pattern in a {\L}uka\-sie\-wicz path corresponds to an occurrence of a red fringe in a red tree.
Again, see Figure~\ref{figRedTreePatterns}, where this is illustrated.

\section{Patterns in {\L}uka\-sie\-wicz paths}\label{sectLukasPatterns}

The asymptotic distribution of patterns in \emph{words} has been investigated before.
For an exposition, see~\cite{FS2009} Examples~I.12, III.26 and IX.13.
The approach taken there makes use of the \emph{correlation polynomial} of a pattern, introduced by
Guibas \& Odlyzko in~\cite{GO1981} to analyse pattern-matching in strings, and also employs the
\emph{cluster method} of Goulden \& Jackson~\cite{GJ2004}.
We refine this approach for use with patterns in
{\L}uka\-sie\-wicz paths by utilising a generalisation of the correlation polynomial
and combining it with an application of the kernel method.

It is readily seen that the bivariate generating function, $L(z,y)$, for
{\L}uka\-sie\-wicz paths, in which $z$ marks length and $y$ marks height, satisfies the functional equation
\begin{equation}\label{eqLukaBGFHeightFuncEq}
  L(z,y) \;=\; z\+y \:+\: \frac{z\+y}{1-y}\big(L(z,1)-y\+L(z,y)\big) .
\end{equation}

Given a pattern $\omega=\omega_1,\ldots,\omega_m$,
let us use $h_i(\omega)=\sum_{j=1}^i \omega_j$ to denote the height after the $i$th step of~$\omega$, and let us call $h_m(\omega)$
the \emph{final height} of $\omega$.

The correlation polynomial of Guibas \& Odlyzko is univariate.
For our purposes, we define the \emph{bivariate} \emph{autocorrelation polynomial}, $\widehat{a}_\omega(z,y)$, for a pattern $\omega=\omega_1,\ldots,\omega_m$ in a {\L}uka\-sie\-wicz path as follows:
$$
\widehat{a}_\omega(z,y) \;=\; \sum_{i=1}^{m-1} c_i\+z^i\+y^{h_i(\omega)} ,
$$
where
$$
c_i  \;=\;
\begin{cases}
1, &\text{if $\omega_{i+1},\ldots,\omega_m = \omega_1,\ldots,\omega_{m-i}$;}\\
0, &\text{otherwise.}
\end{cases}
$$
Thus, $c_i$ records whether $\omega$ matches itself when shifted (left or right) by $i$, the variable $z$ marks the shift, and $y$ marks the height.
For example,
$\widehat{a}_\mathbf{1,1,0,1,1}(z,y) = z^3\+y^2+z^4\+y^3$.

Given a fixed pattern $\omega$ of length $m$ and final height $h$, we want to determine the trivariate generating function, $L_\omega(z,y,u)$, for {\L}uka\-sie\-wicz paths, where $u$ marks the number of occurrences of the pattern~$\omega$ in a path.
In order to achieve this, we first consider the class of {\L}uka\-sie\-wicz paths augmented by distinguishing an arbitrary selection of occurrences of $\omega$. Let $M_\omega(z,y,v)$ be the corresponding generating function, in which $v$ marks distinguished occurrences of the pattern in a path. By the standard inclusion-exclusion principle (see~\cite{FS2009} p.208), we know that
\begin{equation}\label{eqLMInclExcl}
  L_\omega(z,y,u) \;=\; M_\omega(z,y,u-1).
\end{equation}
In order to construct a functional equation for $M_\omega$, we consider subpaths each consisting of a maximal collection of overlapping distinguished occurrences of $\omega$. These collections are called \emph{clusters}.
It is readily seen that the generating function for clusters is
\begin{equation}\label{eqClusterGF}
  C_\omega(z,y,v) \;=\; \frac{z^m\+y^h\+v}{1-v\+\widehat{a}_\omega(z,y)},
\end{equation}
where $v$ is used to mark distinguished occurrences of $\omega$ in a cluster.

Furthermore, we have
\begin{equation}\label{eqMFuncEq}
  M_\omega(z,y,v) \;=\;
  z\+y
  \:+\: \frac{z\+y}{1-y}\big(M_\omega(z,1,v)-y\+M_\omega(z,y,v)\big)
  \:+\: M_\omega(z,y,v) \+ C_\omega(z,y,v),
\end{equation}
since a path grows either by adding an arbitrary step, as in~\eqref{eqLukaBGFHeightFuncEq}, or else by adding a
cluster.\footnote{This equation excludes distinguished occurrences of $\omega$ that begin with the first step of the path; this simplifies the algebra somewhat while having no effect on the asymptotics.}

Combining equations~\eqref{eqLMInclExcl}, \eqref{eqClusterGF} and \eqref{eqMFuncEq} and rearranging
gives us the following functional equation for $L_\omega(z,y,u)$:
$$
L_\omega(z,y,u) \;=\;
\frac
{z\+y\+\big(1 + (1-u)\+\widehat{a}_\omega(z,y)\big)\+\big(1 - y + L_\omega(z,1,u)\big)}
{z^m\+y^h\+(1-y)\+(1-u)+(1-y+z\+y^2)\+\big(1 + (1-u)\+\widehat{a}_\omega(z,y)\big)} .
$$
This equation is susceptible to the kernel method, so $L_\omega(z,1,u)=y_0(z,u)-1$, where $y_0$ is the appropriate root for $y$ of the denominator.
Rearranging, we obtain the following polynomial functional equation for $L=L(z,u)=L_\omega(z,1,u)$,
the bivariate generating function for {\L}uka\-sie\-wicz paths in which $u$ marks occurrences of $\omega$:
\begin{equation}\label{eqLukaBGFPatFuncEq}
  L \;=\; z\+(1+L)^2 \:-\: (1-u)\+\Big(z^m\+L\+(1+L)^h + \big(L - z\+(1+L)^2\big)\+\widehat{a}_\omega(z,1+L)\Big) .
\end{equation}
The fact that $L$ satisfies this equation enables us to demonstrate that patterns in {\L}uka\-sie\-wicz paths are concentrated, and moreover are
distributed normally in the limit.
The following proposition gives very general conditions for this to be the case for some parameter.
\begin{prop}[\cite{FS2009} Proposition~IX.17 with Theorem~IX.12; see also \cite{Drmota1997} Theorem~1]\label{propGaussian}
Let $F(z,u)$ be a bivariate function, analytic at $(0,0)$ and with non-negative Taylor coefficients, and
let $\xi_n$ be the sequence of random variables with probability generating functions
$$
\frac{[z^n]F(z,u)}{[z^n]F(z,1)} .
$$
Assume that $F(z,u)$ is a solution for $y$ of the equation
$$y\;=\;\Phi(z,u,y),$$
where $\Phi$ is a polynomial of degree at least two in $y$,
$\Phi(z,1,y)$ has non-negative Taylor coefficients and is analytic in some domain $|z|<R$ and $|y|<S$,
$\Phi(0,1,0)=0$,
$\Phi_y(0,1,0)\neq1$,
$\Phi_{yy}(z,1,y)\nequiv0$,
and there exist positive $z_0<R$ and $y_0<S$ satisfying the pair of equations
$$
\Phi(z_0,1,y_0)\;=\;y_0, \qquad\quad \Phi_y(z_0,1,y_0)\;=\;0.
$$
Then, as long as its asymptotic variance is non-zero, $\xi_n$
converges in law to a Gaussian distribution with mean and standard deviation asymptotically linear in~$n$.
\end{prop}
All that remains is to check that $L$ satisfies the relevant requirements.
\begin{repthm}{thmLukaPatternsGaussian}
The number of occurrences of a fixed pattern
in a {\L}uka\-sie\-wicz path of length $n$
exhibits
a Gaussian limit distribution
with mean
and standard deviation
asymptotically linear in $n$.
\end{repthm}
\begin{proof}
From~\eqref{eqLukaBGFPatFuncEq}, it can easily be seen that
$L(z,u)$ satisfies the conditions of Proposition~\ref{propGaussian}, with $\Phi(z,1,y)=z\+(1+y)^2$, $z_0=\frac{1}{4}$ and $y_0=1$.
\end{proof}

\section{Summing up}\label{sectLowerBound}

Since patterns in {\L}uka\-sie\-wicz paths are in bijection with
red fringes in red trees,
$L(z,u)$ is also the bivariate generating function
for red trees in which $u$ marks occurrences of the red fringe $\smallF$ corresponding to the pattern $\omega$, with
$m=|\smallF|$ and $h$ the number of components of $\smallF$.
Thus, we know that $\rho_k(\smallF^+)$ is concentrated.
It remains for us to determine the limiting mean.
\begin{prop}\label{propOmegaPlus}
Let $m=|\smallF|$ and $h$ be the number of components of
$\smallF$.
$\rho_k(\smallF^+)$ is concentrated at
$$
\mu_\rho(\smallF^+) \;=\; \frac{1}{2^{2m-h}}.
$$
\end{prop}
\begin{proof}
Let $F(z)=\tfrac{1}{2z} \+ \big(1-\sqrt{1-4\+z}\big)$ be the generating function for plane forests.

Solving~\eqref{eqLukaBGFPatFuncEq} with $u=1$ gives $L(z,1)=F(z)-1$ (as expected).

Similarly, differentiating~\eqref{eqLukaBGFPatFuncEq} with respect to $u$, setting $u=1$, and solving the resulting equation gives
$$
L_u(z,1) \;=\; \frac{z^m\+F(z)^h\+\big(1-(1-2\+z)\+F(z)\big)}{1-4\+z}.
$$
Then, extracting coefficients yields
$$
\begin{array}{lcl}
[z^k]\+ L(z,1) & = & \frac{1}{k+1} \binom{2k}{k}, \\[6pt]
[z^k]\+ L_u(z,1) & = & \binom{2k-2m+h}{k-m-1}.
\end{array}
$$
Hence, applying Proposition~\ref{propMoments}
and taking limits,
$$
\liminfty[k]\mathbb{E}[\rho_k(\smallF^+)] \;=\; \frac{1}{2^{2m-h}}
.
$$
Concentration follows from Theorem~\ref{thmLukaPatternsGaussian}.
\end{proof}

We are finally in a position to compute a lower bound for the growth rate of the class of permutations avoiding $\pdiamond$, proving our main theorem.
\begin{repthm}{thm1324LowerBound}
$\gr(\av(\pdiamond)) > 9.81$.
\end{repthm}
\begin{proof}
We calculate the contribution to the growth rate from pairs consisting of a tree and a forest of bounded size.
From Proposition~\ref{propLB}, we know that, for each $N>0$, the growth rate is at least
$$
g_N(\lambda,\delta)
\;=\;
E(\lambda,\delta)^{\nfrac{1}{(1+\lambda)}}
\:\times\:
\prod\limits_{|\subT|+|\subF|\leqslant N}
\!\!
\,
Q(\smallT,\smallF)^{2\+\delta\+\lambda\+\mu(\subT,\subF)/(1+\lambda)},
$$
where
$$
\mu(\smallT,\smallF) \;=\;
\mu_\beta(\smallT)\+
\big(
\mu_\gamma(|\smallF|)\+\mu_\rho(\smallF^+)  \:+\:
\mu_\gamma(>\!|\smallF|)\+\mu_\rho(\smallF)
\big) ,
$$
as follows from~\eqref{eqChi} and~\eqref{eqZeta} and Propositions~\ref{propPsi}, \ref{propPhi}, \ref{propOmega} and~\ref{propOmegaPlus}.

Using Mathematica~\cite{Mathematica} to evaluate $Q(\smallT,\smallF)$ and $\mu(\smallT,\smallF)$ and then to apply numerical maximisation over values of $\lambda$ and $\delta$ yields
$$
g_{14}(\lambda,\delta)
\;>\;
9.81056
$$
with $\lambda\approx0.69706$ and $\delta\approx0.75887$.
\end{proof}
The determination of this value requires the processing of more than 1.6 million pairs consisting of a tree and a forest.
Larger values of $N$ would require more sophisticated programming techniques.
However, increasing $N$ is unlikely to lead to a significantly improved lower bound;
although the rate of convergence at $N\!=\!14$ is still quite slow, numerical analysis of the computational data suggests that $\liminfty[N]\max\limits_{\lambda,\delta}g_N(\lambda,\delta)$ is probably not far from 9.82.

We conclude with the observation that in the construction that gives our bound, the mean number of vertices in a blue subtree, $1/\delta$, is less than 1.32.
We noted earlier that
the cigar-shaped boundary regions
of
a typical $\pdiamond$-avoider
contain numerous small subtrees (although it is not immediately obvious how one should identify such a boundary tree).
Is it the case that
the mean size of these
subtrees is asymptotically bounded?
Perhaps, on the contrary, their average size grows unboundedly (but very slowly), and understanding how (and the rate at which) this occurs would lead to an improved lower bound.
In the meantime, the following question might be somewhat easier to answer:
\begin{question}
Asymptotically, what proportion of the points in a typical $\pdiamond$-avoider are left-to-right minima or right-to-left maxima?
\end{question}

\subsubsection*{Acknowledgements}

The author would like to thank Mireille Bousquet-M\'elou for valuable discussions relating to this work during the Cardiff Workshop on Combinatorial Physics in December 2013.
He is also grateful to Robert Brignall who provided useful feedback which helped to improve the presentation.

\emph{S.D.G.}

\bibliographystyle{plain}
{\footnotesize\bibliography{mybib}}

\end{document}